\documentclass[12pt]{article}
\pdfoutput=1

\usepackage{amsmath, amsthm, amssymb}
\usepackage{fullpage}
\usepackage{enumerate}
\usepackage{ifpdf}
\usepackage{mathrsfs}
\ifpdf
\usepackage[pdftex]{graphicx}
\else
\usepackage[dvips]{graphicx}
\fi
\usepackage{tikz}
\usepackage[all]{xy}
\usetikzlibrary{arrows,chains,matrix,positioning,scopes}
\makeatletter
\tikzset{join/.code=\tikzset{after node path={%
\ifx\tikzchainprevious\pgfutil@empty\else(\tikzchainprevious)%
edge[every join]#1(\tikzchaincurrent)\fi}}}
\makeatother
\tikzset{>=stealth',every on chain/.append style={join},
         every join/.style={->}}
\tikzstyle{labeled}=[execute at begin node=$\scriptstyle,
   execute at end node=$]

\input xy
\xyoption{all}

\usepackage[pdftex,plainpages=false,hypertexnames=false,pdfpagelabels]{hyperref}
\newcommand{\arXiv}[1]{\href{http://arxiv.org/abs/#1}{\tt arXiv:\nolinkurl{#1}}}
\newcommand{\arxiv}[1]{\href{http://arxiv.org/abs/#1}{\tt arXiv:\nolinkurl{#1}}}

\newcommand{\doi}[1]{\href{http://dx.doi.org/#1}{{\tt DOI:#1}}}

\newcommand{\googlebooks}[1]{(preview at \href{http://books.google.com/books?id=#1}{google books})}

\usepackage{xcolor}
\definecolor{dark-red}{rgb}{0.7,0.25,0.25}
\definecolor{dark-blue}{rgb}{0.15,0.15,0.55}
\definecolor{medium-blue}{rgb}{0,0,.8}
\definecolor{DarkGreen}{RGB}{0,150,0}
\hypersetup{
   colorlinks, linkcolor={purple},
   citecolor={medium-blue}, urlcolor={medium-blue}
}

\theoremstyle{plain}
\newtheorem{thm}{Theorem}[section]
\newtheorem*{thm*}{Theorem}
\newtheorem{alphathm}{Theorem}

\newtheorem{cor}[thm]{Corollary}
\newtheorem*{cor*}{Corollary}
\newtheorem{lem}[thm]{Lemma}

\newtheorem*{quest*}{Question}

\theoremstyle{definition}
\newtheorem{defn}[thm]{Definition}

\newtheorem{nota}[thm]{Notation}

\newtheorem{rem}[thm]{Remark}

\DeclareMathOperator{\Aut}{Aut}

\DeclareMathOperator{\Gr}{Gr}

\DeclareMathOperator{\id}{id}

\DeclareMathOperator{\op}{op}

\DeclareMathOperator{\Tr}{Tr}
\DeclareMathOperator{\tr}{tr}
\DeclareMathOperator{\idd}{id}
\DeclareMathOperator{\ran}{range}
\DeclareMathOperator{\coeff}{coeff}

\newcommand{\comment}[1]{}

\newcommand{\be}{\begin{enumerate}[(1)]}
\newcommand{\ee}{\end{enumerate}}

\newcommand{\N}{\mathbb{N}}
\newcommand{\Z}{\mathbb{Z}}

\newcommand{\F}{\mathbb{F}}
\newcommand{\R}{\mathbb{R}}

\newcommand{\C}{\mathbb{C}}

\newcommand{\set}[2]{\left\{#1 \middle| #2\right\}}

\newcommand{\e}{\epsilon}

\newcommand{\noshow}[1]{}
\newcommand{\MR}[1]{}

\newcommand{\Asterisk}{\mathop{\scalebox{1.5}{\raisebox{-0.2ex}{$*$}}}}%
\newcommand{\p}{\partial}

\newcommand{\TL}{\cT\hspace{-.08cm}\cL}

\def\semicolon{;}
\def\applytolist#1{
    \expandafter\def\csname multi#1\endcsname##1{
        \def\multiack{##1}\ifx\multiack\semicolon
            \def\next{\relax}
        \else
            \csname #1\endcsname{##1}
            \def\next{\csname multi#1\endcsname}
        \fi
        \next}
    \csname multi#1\endcsname}

\def\calc#1{\expandafter\def\csname c#1\endcsname{{\mathcal #1}}}
\applytolist{calc}QWERTYUIOPLKJHGFDSAZXCVBNM;
\def\bbc#1{\expandafter\def\csname bb#1\endcsname{{\mathbb #1}}}
\applytolist{bbc}QWERTYUIOPLKJHGFDSAZXCVBNM;
\def\bfc#1{\expandafter\def\csname bf#1\endcsname{{\mathbf #1}}}
\applytolist{bfc}QWERTYUIOPLKJHGFDSAZXCVBNM;
\def\sfc#1{\expandafter\def\csname s#1\endcsname{{\sf #1}}}
\applytolist{sfc}QWERTYUIOPLKJHGFDSAZXCVBNM;
\def\ffc#1{\expandafter\def\csname f#1\endcsname{{\mathfrak #1}}}
\applytolist{ffc}QWERTYUIOPLKJHGFDSAZXCVBNM;

\usepackage{yfonts}

\tikzstyle{shaded}=[fill=red!10!blue!20!gray!30!white]
\tikzstyle{unshaded}=[fill=white]
\tikzstyle{empty box}=[circle, draw, thick, fill=white, opaque, inner sep=2mm]
\tikzstyle{annular}=[scale=.7, inner sep=1mm, baseline]
\tikzstyle{rectangular}=[scale=.75, inner sep=1mm, baseline=-.1cm]
\usetikzlibrary{calc}

\newcommand{\nbox}[6]{
	\draw[thick, #1] ($#2+(-#3,-#3)+(-#4,0)$) rectangle ($#2+(#3,#3)+(#5,0)$);
	\coordinate (ZZa) at ($#2+(-#4,0)$);
	\coordinate (ZZb) at ($#2+(#5,0)$);
	\node at ($1/2*(ZZa)+1/2*(ZZb)$) {#6};
}

\begin{document}

\title{Free product C$^{*}$-algebras associated to graphs, free differentials, and laws of loops}
\author{Michael Hartglass}
\date{\today}
\maketitle
\begin{abstract}
We study a canonical C$^*$-algebra, $\cS(\Gamma, \mu)$, that arises from a weighted graph $(\Gamma, \mu)$, specific cases of which were previously studied in the context of planar algebras.  We discuss necessary and sufficient conditions of the weighting which ensure simplicity and uniqueness of trace of $\cS(\Gamma, \mu)$, and study the structure of its positive cone.  We then study the $*$-algebra, $\cA$, generated by the generators of $\cS(\Gamma, \mu)$, and use a free differential calculus and techniques of Charlesworth and Shlyakhtenko, as well as Mai, Speicher, and Weber to show that certain ``loop" elements have no atoms in their spectral measure.  After modifying techniques of Shlyakhtenko and Skoufranis to show that self adjoint elements $x \in M_{n}(\cA)$ have algebraic Cauchy transform, we explore some applications to eigenvalues of polynomials in Wishart matrices and to diagrammatic elements in von Neumann algebras initially considered by Guionnet, Jones, and Shlyakhtenko.
\end{abstract}


\section{Introduction}

The study of graphs and has been a prevailing force in the development of many areas of operator algebras.  Cuntz and Krieger \cite{MR561974} initiated the study of graph C$^{*}$-algebras, which are canonical C$^*$-algebras that are formed from directed graphs.  The principal graph of a planar algebra or subfactor  \cite{MR696688, MR1334479, MR1424954} has played an important role in the development and study of subfactor theory. 

Given a standard invariant of a subfactor $\cP_{\bullet}$, Popa reconstructed a Jones tower (see \cite{MR696688} for a definition) of II$_{1}$ factors whose standard invariant is $\cP_{\bullet}$ \cite{MR1334479}.  More recently, Guionnet, Jones, and Shlyakhtenko (GJS) gave a diagrammatic construction of a Jones' tower $\cM_{0} \subset \cM_{1} \subset \cM_{2} \subset \cdots$ of II$_{1}$ factors whose standard invariant is $\cP_{\bullet}$ \cite{MR2732052}.  In subsequent work, GJS constructed a (semi)finite factor by applying Shlyakhtenko's operator-valued generalization \cite{MR1704661} of Voiculescu's free Gaussian functor \cite{MR1217253} to obtain the isomorphism class of the $\cM_{k}$ in the case that $\cP_{\bullet}$ is finite-depth \cite{MR2807103}.  More specifically, if $\cP_{\bullet}$ is finite-depth with finite principal graph $\Gamma$, GJS constructed a finite von Neumann algebra, $\cM(\Gamma)$ by assigning to each edge $e$ of $\Gamma$, an $\ell^{\infty}(V(\Gamma))$-valued semicircular element $X_{e}$.  The factor $\cM_{0}$ was realized as a of corner $\cM(\Gamma)$ which was shown to be an interpolated free group factor.  

In later work \cite{MR3110503, 1208.5505}, the author developed a canonical free-product von Neumann algebra, $\cM(\Gamma, \mu)$ that one can associate to an arbitrary undirected, weighted, graph with weighting $\mu: V(\Gamma) \rightarrow \R^{+}$.  The key observation in \cite{MR3110503} is that if $(\Gamma, \mu)$ is a subgraph of $(\Gamma', \mu)$, then there is a canonical, possibly nonunital, inclusion $\cM(\Gamma, \mu) \rightarrow \cM(\Gamma', 
\mu)$, compressions of which are a standard embedding in the sense of \cite{MR1201693}.  This observation showed that if a planar algebra $\cP_{\bullet}$ is infinite depth, then in the construction in \cite{MR2732052}, $\cM_{k} \cong L(\F_{\infty})$ for all $k$.

In an effort to study the diagrammatic C$^{*}$-algebras ($\cA_{0} \subset \cA_{1} \subset \cA_{2} \subset \cdots$) that arise in the construction in \cite{MR2732052}, the author and Penneys defined the C$^{*}$-algebra analogue of $\cM(\Gamma, \mu)$, called $\cS(\Gamma, \mu)$ \cite{1401.2485, MR3266249}.   Although this algebra was defined for an arbitrary weighting $\mu$, most properties of $\cS(\Gamma, \mu)$ were studied only in the case of $\Gamma$ being the principal graph of a planar algebra with $\mu$ the associated weighting.   In this case, the authors proved that  $\cS(\Gamma, \mu)$ is simple, has unique trace, established $KK$-equivalence between $C_{0}(V(\Gamma))$ and $\cS(\Gamma, \mu)$.  Utilizing a Morita equivalence between $\cA_{k}$ and $\cS(\Gamma, \mu)$, the same properties were deduced for the algebras $\cA_{k}$.

In section \ref{sec:free} we will study $\cS(\Gamma, \mu)$ in the setting of an arbitrary connected undirected graph $\Gamma$ and arbitrary weighting $\mu$ on $V(\Gamma)$.  This C$^{*}$-algebra is generated by $C_{0}(V(\Gamma))$ (spanned by the indicator functions $\set{p_{\alpha}}{\alpha \in V(\Gamma)}$) and $C_{0}(V(\Gamma))$-valued semicircular elements $X_{e}$ associated to each edge $e$.  In the von Neumann algebra case, $\cM(\Gamma, \mu)$, the following was proven in \cite{MR3110503}.

\begin{thm*}[\cite{MR3110503}]
Suppose $\Gamma$ is a finite, connected, unoriented graph with at least two edges.  If $\alpha,\beta \in V(\Gamma)$, then write $\alpha \sim \beta$ if $\alpha$ and $\beta$ are joined by at least one edge, and let $n_{\alpha, \beta}$ be the number of edges joining with $\alpha$ and $\beta$ as endpoints.  Finally, let $V_{>}$ be the set of vertices, $\beta$ satisfying $\mu(\beta) > \sum_{\alpha \sim \beta} n_{\alpha, \beta}\mu(\alpha)$.  We have
$$
\cM(\Gamma, \mu) \cong L(\F_{t}) \oplus \bigoplus_{\gamma \in V_{>}} \overset{r_{\gamma}}{\C}
$$
where $r_{\gamma} \leq p_{\gamma}$ and $\tr(r_{\gamma}) = \mu(\alpha) - \sum_{\alpha \sim \gamma} n_{\alpha, \beta}\mu(\alpha)$ (see Notation \ref{nota:free} for the meaning of the direct sum notation).  Moreover, the parameter, $t$, can be computed using Dykema's ``free dimension" formulas \cite{MR1201693, MR3164718}.  In particular, $\cM(\Gamma, \mu)$ is a factor if and only if $V_{>}$ is non-empty.
\end{thm*}  
We prove the C$^{*}$-algebra analogue.
\begin{alphathm}
Let $\Gamma$ and $V_{>}$ be as in the statement of the previous theorem .  Let $V_{=}$ be the set of vertices, $\beta$ satisfying $\mu(\beta) = \sum_{\alpha \sim \beta} n_{\alpha, \beta}\mu(\alpha)$ and $V_{\geq} = V_{>} \cup V_{=}$.  Let $I$ be the norm-closed ideal generated by some $p_{\alpha}$ with $\alpha \in V(\Gamma) \setminus V_{\geq}$. Then $I$ contains $\set{p_{\beta}}{\beta \in V(\Gamma) \setminus V_{\geq}}$ and does not intersect $\set{p_{\gamma}}{\gamma \in V_{\geq}}$.  In addition, $I$ is generated by $\set{X_{e}}{e \in E(\Gamma)}$.  Furthermore, we have
\be  

\item $I$ is simple, has unique tracial state, and has stable rank 1.

\item $I$ is unital if and only if $V_{=}$ is empty.   If $V_{=}$ is empty, then
$$
\cS(\Gamma, \mu) = I \oplus \bigoplus_{\gamma \in V_{>}} \overset{r_{\gamma}}{\C}
$$
with $r_{\gamma} \leq p_{\gamma}$ and $\tr(r_{\gamma}) = \mu(\alpha) - \sum_{\alpha \sim \gamma} n_{\alpha, \beta}\mu(\alpha)$. If $V_{=}$ is not empty, then 
$$
\cS(\Gamma, \mu) = \mathcal{I} \oplus \bigoplus_{\gamma \in V_{>}} \overset{r_{\gamma}}{\C}
$$
where $\mathcal{I}$ is unital, the strong operator closures of $I$ and $\mathcal{I}$ coincide in $L^{2}(\cS(\Gamma, \mu), \Tr)$, and $\mathcal{I}/I \cong \bigoplus_{\beta \in V_{=}} \C$. 

\item $K_{0}(I) \cong \Z\set{[p_{\beta}]}{\beta \in V\setminus V_{\geq}} \text{ and } K_{1}(I) = \{0\}$ where the first group is the free abelian group on the classes of projections $[p_{\beta}]$.  Furthermore, $K_{0}(I)^{+} = \set{x \in K_{0}(I)}{\Tr(x) > 0} \cup \{0\}$.
\ee

In particular, $\cS(\Gamma, \mu)$ is simple with unique tracial state if and only if $V_{\geq}$ is empty.
\end{alphathm}
A version of this theorem for infinite $\Gamma$ is given in Corollary \ref{cor:infinite} below. 

Section \ref{sec:poly} switches gears and studies elements of $\cA$, the $*$-algebra generated by $\set{p_{\alpha}}{\alpha \in V(\Gamma)} \cup \set{X_{e}}{e \in E(\Gamma)}$.  We develop a free differential calculus along the lines of Voiculescu \cite{MR1228526}.  Using modifications of the methods in \cite{1408.0580, 1407.5715}, we are able to prove the following theorem.

\begin{alphathm}
Let $Q \in \cA$ be nonzero and satisfying $Qp_{\alpha} = Q$ and $\mu(\alpha) = \min\set{\mu(\beta)}{\beta \in V(\Gamma)}$  If $a = a^{*} \in \cM(\Gamma, \mu)$, $ap_{\alpha} = a$, and $Qa = 0$, then $a = 0$.  In particular, if $Q = Q^{*}$, then the law of $Q$ with respect to $\Tr$ has no atoms.
\end{alphathm}
Section \ref{sec:poly} also uses the techniques of \cite{MR3356937} to prove the following theorem
\begin{alphathm}
Let $x \in M_{n}(\cA)$ be self-adjoint.  Then the Cauchy transform of $x$ is algebraic
\end{alphathm}
As a result, further deductions can be made about the laws of self-adjoint elements in $M_{n}(\cA)$.  See Corollary \ref{cor:alg} and Theorem \ref{thm:entropy} below.

Section \ref{sec:app} is devoted to applications of the previous sections to the limiting laws of polynomials in Wishart matrices and to diagrammatic elements in the construction in \cite{MR2732052}.  Much like the observations about Gaussian unitary ensembles in \cite{MR3356937}, it is shown that eigenvalues in self-adjoint polynomials of Wishhart matrices also experience a weak repulsion.

\subsection{Acknowledgements}

The author would like to thank Ken Dykema, Brent Nelson, David Penneys, and Dima Shlyakhtenko  for helpful conversations.  The author is especially indebted to Dima Shlyakhtenko for conversations relating the methods in \cite{1408.0580}.


\section{The free graph algebra}\label{sec:free}

\subsection{Constructing the algebra}

We will start with countable, weighted, connected, undirected, locally finite graph, $\Gamma = (\Gamma, V, E, \mu)$ where $V$ is the vertex set of $\Gamma$, $E$ is the edge set, and $\mu : V \rightarrow \R_{>0}$ is a strictly positive function on the vertices of $\Gamma$.  We will now form the directed version of $\Gamma$ as in \cite{1401.2485}.

\begin{defn}\label{defn:direct}
If $\Gamma$ is an undirected graph, we will form the directed version of $\Gamma$, $\vec{\Gamma} = (\vec{\Gamma}, \vec{V}, \vec{E}, \vec{\mu}, s, t)$ as the weighted directed graph with the following properties.
\begin{enumerate}

\item $\vec{V} = V$ and $\vec{\mu} = \mu$.

\item $\vec{E}$ is constructed as follows:  For each $e \in E(\Gamma)$ having $\alpha \neq \beta$ as its two endpoints, there are two edges $\e$ and $\e^{op}$ satisfying 
$$
s(\e) = t(\e^{op}) = \alpha \text{ and } t(\e) = s(\e^{op}) = \beta . 
$$
For each $e \in E(\Gamma)$ which is a loop at a vertex, $\gamma$, there is one loop (which we will denote as $\e$) in $E(\vec{\Gamma})$ at $\gamma$.

\end{enumerate}
\end{defn}

Notice that there is an involution $\e \mapsto \e^{op}$ on $\vec{E}$ provided that $\e = \e^{op}$ whenever $\e$ is a loop.  Let $X(\vec{\Gamma})$ be the complex vector space with basis $\vec{E}$, and $C_{fin}(V)$ be the finitely supported complex-valued functions on $V$.  The algebra $C_{fin}(V)$ is spanned by the elements $p_{\alpha}$ for $\alpha \in \Gamma$ satisfying $p_{\alpha}(\beta) = \delta_{\alpha, \beta}$. 

There is a natural $C_{fin}(V)$ bimodule structure on $X(\Gamma)$ which is given by
$$
p_{\alpha}\e = \delta_{s(\e), \alpha}p_{\alpha} \text{ and } \e p_{\alpha} = \delta_{t(\e), \alpha}p_{\alpha}.
$$ 
With these actions in mind, $X(\Gamma)$ exhibits a $C_{fin}(V)$-valued inner product $\langle \cdot | \cdot \rangle$ which is given by 
$$
\langle \e'|\e \rangle = \delta_{\e, \e'}p_{t(\e)}
$$
with linearity in the \emph{second} variable and conjugate-linearity in the first.  Let $\cC$ denote $C_{0}(V)$.  Using this inner product, $X(\vec{\Gamma})$ completes to a $\cC-\cC$ bimodule $\cX(\vec{\Gamma})$ with inner product $\langle \cdot|\cdot\rangle_{\cC}$.  We will denote this Hilbert bimodule as $\cX$ when the context is clear.  

We now form the \underline{full Fock space} of $\vec{\Gamma}$, $\cF(\vec{\Gamma})$ which is given by
$$
\cF(\cX) = \cC \oplus \bigoplus_{n \geq 1} \cX^{\bigotimes^{n}_{\cC}}
$$
where $\otimes_{\cC}$ denotes the internal tensor product of $\cC-\cC$ Hilbert bimodules.  $\cC$, together with elements of the form $\e_{1} \otimes \e_{2} \otimes \cdots \otimes \e_{n}$ span a dense subset of $\cF(\cX)$.  Note that the elementary tensor $\e_{1} \otimes \e_{2} \otimes \cdots \otimes \e_{n}$ is zero unless the edges form a path in $\vec{\Gamma}$ i.e. $t(\e_{k}) = s(\e_{k+1})$ for $1 \leq k < n$.

We have the usual creation operator $\ell(\xi)$ for $\xi \in \cX$ which is defined by
$$
\ell(\xi)p_{\alpha} = \xi\cdot p_{\alpha} \text{ and } \ell(\xi)(\xi_{1} \otimes \cdots \otimes \xi_{n}) = \xi \otimes \xi_{1} \otimes \cdots \otimes \xi_{n}.
$$
The creation operator $\ell(\xi)$ is bounded and adjointable with adjoint $\ell(\xi)^{*}$ given by
$$
\ell(\xi)^{*}p_{\alpha} = 0 \text{ and } \ell(\xi)^{*}(\xi_{1} \otimes \xi_{2} \otimes \cdots \otimes \xi_{n}) = \langle \xi|\xi_{1}\rangle_{\cC} \xi_{2} \otimes \cdots \otimes \xi_{n}.
$$
We define the \underline{Pimsner-Topelitz algebra of $\vec{\Gamma}$}, $\cT(\vec{\Gamma})$ \cite{MR1426840} to be the C$^{*}$-algebra generated by $\cC$ and $\{\ell(\xi)| \xi \in \cX\}$.  

We will now construct our finite algebra associated to $\Gamma$ with weighting $\mu$.  
\begin{defn}
Let $\Gamma = (\Gamma, V, E, \mu)$ with $\vec{\Gamma}$ as in definition \ref{defn:direct}.  
\begin{enumerate}

\item If $e \in E$ is a loop in $\Gamma$, and $\e$ is the associated loop in $\vec{E}$, we define $X_{e} = \ell(\e) + \ell(\e)^{*}$

\item If $e \in E$ has $\alpha \neq \beta$ as its endpoints, let $\e$ and $\e^{op}$ be the two associated edges as in Definition \ref{defn:direct}, with $s(\e) = t(\e^{op}) = \alpha$ and $t(\e) = s(\e^{op}) = \beta$ .  Also set $a_{\e} = \sqrt[4]{\frac{\mu(\alpha)}{\mu(\beta)}}$ .  Then we define 
$$
X_{e} = a_{\e}\ell(\e) + a_{\e}^{-1}\ell^{*}(\e^{op}) + a_{\e}^{-1}\ell(\e^{op}) + a_{\e}\ell^{*}(\e).
$$

\end{enumerate}

We define the \underline{free graph algebra associated to $\Gamma, \mu$}, $\cS(\Gamma, \mu)$, to be the subalgebra of $\cT(\Gamma)$ generated by $\cC$ and the elements $(X_{e})_{e \in E}$ 

\begin{rem}\label{rem:opcirc}
 The elements $X_{e}$ are always self adjoint.  Note that if $e$ has $\alpha \neq \beta$ as endpoints, then 
$$
p_{\alpha}X_{e}p_{\beta} = a_{\e}\ell(\e) + a_{\e}^{-1}\ell(\e^{op})^{*} \text{ and } p_{\beta}X_{e}p_{\alpha} = a_{\e}^{-1}\ell(\e^{op}) + a_{\e}\ell(\e)^{*}
$$
We will set $X_{\e} = p_{\alpha}X_{e}p_{\beta}$ and $X_{\e^{op}} = p_{\beta}X_{e}p_{\alpha} $.
\end{rem}

As in \cite{MR1704661,1401.2485} there is a conditional expectation $E: S(\Gamma, \mu) \rightarrow \cC$ which is given by $E(x) = PxP = \sum_{\alpha \in V} \langle p_{\alpha}| xp_{\alpha}\rangle_{\cC}$ with $P: \cF(\cX) \rightarrow \cC$ the orthogonal projection.  We have the following theorem.

\begin{thm}[\cite{1401.2485}]\label{thm:freeexpectation}
\begin{enumerate}

\item The expectation $E$ is faithful on $\cS(\Gamma, \mu)$.  Moreover, if we let $\tr$ be the (semi)-finite trace on $\cC$ determined by $\tr(p_{v}) = \mu(v)$, then $\Tr = \tr\circ E$ is a (semi)-finite tracial weight on $\cS(\Gamma, \mu)$.

\item The elements $(X_{e})_{e \in E(\Gamma)}$ are free with amalgamation over $\cC$ with respect to $E$.

\item If $\e$ has $s(\e) = \alpha$ and $t(\e) = \beta$ and $X_{\e}$ is as in Remark \ref{rem:opcirc} then the law of $X_{\e}^{*}X_{\e}$ in $p_{\beta}\cS(\Gamma, \mu){\e}p_{\beta}$ with respect to $\Tr(p_{\beta} \cdot p_{\beta})$ is a Free Poisson with the following distribution

\begin{itemize}

\item $\mu(\beta)\frac{\sqrt{4a^{2}x - (a^{4} - 1 - a^{2}x)^{2}}}{2\pi x}{\bf 1}_{[a^{2} + a^{-2} - 2, a^{2} + a^{-2} + 2]}\, dx$ if $\mu(\beta) \leq \mu(\alpha)$ with $a = a_{\e}$.

\item $(\mu(\beta) - \mu(\alpha))\delta_{0} + \mu(\alpha)\frac{\sqrt{4a^{2}x - (a^{4} - 1 - a^{2}x)^{2}}}{2\pi x} {\bf 1}_{[a^{2} + a^{-2} - 2, a^{2} + a^{-2} + 2]}\, dx$ if $\mu(\beta) \geq \mu(\alpha)$
with $a = a_{\e}$.

\end{itemize} 

\end{enumerate}
\end{thm}

\end{defn}

\begin{rem}\label{rem:polar}
Notice that $\Tr(X_\e^{*}X_{\e}) = \sqrt{\mu(s(\e))\mu(t(\e))}$.  If $\mu(\alpha) \neq \mu(\beta)$, then $a_{\e}^{2} + a_{\e}^{-2} - 2 > 0$ so we conclude that the polar part of $X_{\e}$ is in $\cS(\Gamma, \mu)$.  Therefore, if $\mu(\alpha) < \mu(\beta)$ then $p_{\alpha}$ is equivalent to a subprojection of $p_{\beta}$ in $\cS(\Gamma, \mu)$.  It is important to note that the polar part of $X_{\e}$ will not be in $\cS(\Gamma, \mu)$ if $\mu(\alpha) = \mu(\beta)$.  This observation is consistent with Corollary \ref{cor:K-theory} below.
\end{rem}

We will also need the following lemma, to be used in Section \ref{sec:alg}.

\begin{lem}\label{lem:rec}
Suppose that $\e_{1}, \dots, \e_{n}$ is a loop in in $\vec{\Gamma}$ of length at least 3.  Then
\begin{align*}
\Tr(X_{\e_{1}}\cdots X_{\e_{n}}) &= \frac{1}{\sqrt{\mu(s(\e_{n}))\mu(t(\e_{n}))}}\sum_{\substack{\e_{j} = \e_{n}^{\op} \\ j \not\in \{1, n-1\}}} \Tr(X_{\e_{1}}\cdots X_{\e_{j-1}})\cdot \Tr(X_{\e_{j+1}}\cdots X_{\e_{n-1}})\\
&+ \delta_{\e_{n-1}, \e_{n}^{op}}\sqrt{\frac{\mu(s(\e_{n}))}{\mu(t(\e_{n}))}}\Tr(X_{\e_{1}}\cdots X_{\e_{n-2}}) + \delta_{\e_{1}, \e_{n}^{op}}\sqrt{\frac{\mu(t(\e_{n}))}{\mu(s(\e_{n}))}}\Tr(X_{\e_{2}}\cdots X_{\e_{n-1}})
\end{align*}

\end{lem}

\begin{proof}
Noting that $\Tr = \tr \circ E$, we will first compute the $\cC$-valued inner product $\langle  \widehat{p_{t(\e_{n})}} | X_{\e_{1}}\cdots X_{\e_{n}}\widehat{p_{t(\e_{n})}} \rangle$.  To this end, we note that 
$$
X_{\e_{1}}\cdots X_{\e_{n}}\widehat{p_{t(\e_{n})}} =a_{\e_{n}}X_{\e_{1}}\cdots X_{\e_{n-1}} \e_{n}.
$$
Clearly for $X_{\e_{1}}\cdots X_{\e_{n-1}} \e_{n}$ to have a $\widehat{p_{t(\e_{n})}}$ component, it must be the case that at least one element in $\{\e_{1}, \dots, \e_{n-1}\}$ must be $\e_{n}^{\op}$.  By expanding each $X_{\e_{j}} = a_{\e_{j}}\ell(\e_{j}) + a_{\e_{j}}^{-1}\ell(\e_{j})^{*}$, we will count how many terms in $a_{\e_{n}}X_{\e_{1}}\cdots X_{\e_{n-1}} \e_{n}$ have a $\widehat{p_{t(\e_{n})}}$ component.

Assuming that there is at least one $j \in \{1, 2, \dots , n-1\}$ satisfying $\e_{j} = \e_{n}^{\op}$, we see that
\begin{align*}
\langle \widehat{p_{t(\e_{n})}} | a_{\e_{n}}X_{\e_{1}}\cdots X_{\e_{n-1}} \e_{n} \rangle &= \sum_{\e_{j} = \e_{n}^{\op}} \sum_{\substack{T_{k} \in \{a_{\e_{k}}\ell(\e_{k}), a_{\e_{k}}^{-1}\ell(\e_{k}^{\op})^{*}\} \\ T_{k}T_{k+1}\cdots T_{n-1}\e_{n} \not\in \C \widehat{p_{t(\e_{n})}}}} E(a_{\e_{n}}X_{\e_{1}}\cdots X_{\e_{j-1}})\langle \widehat{p_{t(\e_{n})}} | a_{\e_{n}}\ell(e_{j}^{\op})^{*}T_{j+1}\cdots T_{n-1}\e_{n}\rangle\\
&= \sum_{\e_{j} = \e_{n}^{\op}}  a_{\e_{n}}^{2} E(X_{\e_{1}}\cdots X_{\e_{j-1}})\ell(\e_{n})^{*}E(X_{j+1}\cdots X_{n-1})\e_{n}\\
&= \sum_{\substack{\e_{j} = \e_{n}^{\op} \\ j \not\in\{1, n-1\}}} a_{\e_{n}}^{2} \frac{\Tr(X_{\e_{1}}\cdots X_{\e_{j-1}})}{\Tr(p_{t(\e_{n})})} \cdot \frac{\Tr(X_{j+1}\cdots X_{n-1})}{\Tr(p_{s(\e_{n})})} \widehat{p_{t(\e_{n})}}\\
&+ \delta_{\e_{n-1}, \e_{n}^{\op}} a_{\e_{n}}^{2} \frac{\Tr(X_{\e_{1}}\cdots X_{\e_{n-2}})}{\Tr(p_{t(\e_{n})})}\widehat{p_{t(\e_{n})}} +  \delta_{\e_{1}, \e_{n}^{\op}} a_{\e_{n}}^{2} \frac{\Tr(X_{\e_{2}}\cdots X_{\e_{n-1}})}{\Tr(p_{s(\e_{n})})}\widehat{p_{t(\e_{n})}}. 
\end{align*} 
Taking the trace of both sides gives the desired formula.
\end{proof}


\subsection{$KK$-groups}

We will begin by computing the $KK$-groups of $\cS(\Gamma, \mu)$.  This computation was done in \cite{1401.2485} but for completeness, we will present the argument here.  If we let $\cB$ be a separable C$^{*}$-algebra and $\cY$ a countably generated Hilbert $\cB-\cB$ bimodule, we form the Fock space of $\cY$,
$$
\cF(\cY) = \cB \oplus \bigoplus_{n\geq 1} \cY^{\bigotimes_{\cB}^{n}}.
$$
We define the Pimsner-Toeplitz algebra of $\cY$, $\cT(\cY)$ to be C$^{*}$-algebra generated by the operators $\cB$ and $\ell(\xi)$ for $\xi \in \cY$ where 
$$
\ell(\xi)b = \xi\cdot b \text{ and } \ell(\xi)(\xi_{1} \otimes \cdots \otimes \xi_{n}) = \xi \otimes \xi_{1} \otimes \cdots \otimes \xi_{n}.
$$
for all $b \in \cB$ and $\xi_{1}, \dots, \xi_{n} \in \cY$.  A result of Pimsner says the following.
\begin{thm}[\cite{MR1426840}]\label{thm:Pimsner}
Assume that $\cB$ acts on $\cY$ by compact operators.  Then the inclusion $\cB \hookrightarrow \cT(\cY)$ is a $KK$-equivalence.
\end{thm}

In \cite{Germain}, one chooses a closed real subspace $\cY_{\R}$ of $\cY$ with the property that $\overline{\cB\cY_{\R}\cB} = \cY$, and defines the algebra $\cS(\cY_{\R})$to be the C$^{*}$-algebra generated by $\cB$ and $\{\ell(\xi) + \ell(\xi)^{*}| \xi \in \cY_{\R}\}$.  By observing that the homotopy in \cite{MR1426840} leaves the subspace $\cF(\cY) \otimes_{\cB}\cS(\cY_{\R})$ of $\cF(\cY) \otimes_{\cB}\cT(\cY)$ invariant, the following theorem was proven.

\begin{thm}[\cite{Germain}]\label{thm:Germain}
Assume that $\cB$ acts on $\cY$ by compact operators.  The inclusions $\cB \hookrightarrow \cS(\cY_{\R}) \hookrightarrow \cT(\cY)$ are $KK$-equivalences.
\end{thm} 

  We will let $\cX_{\R, \mu}$ be the closure of the real subspace spanned by the elements $a_{\e}\e + a_{\e}^{-1}\e^{op}$.  We observe that $\cS(\Gamma, \mu)$ is generated by $\cC$ and the elements $\ell(\xi) + \ell(\xi)^{*}$ for $\xi \in \cX_{\R, \mu}$.  Using this realization, Theorem \ref{thm:Germain} gives the following corollary.

\begin{cor}\label{cor:K-theory}
Since $\Gamma$ is locally finite, $\cC$ acts on $\cX$ by compact operators.  Therefore, the inclusions $\cC \hookrightarrow \cS(\Gamma, \mu) \hookrightarrow \cT(\vec{\Gamma})$ are $KK$-equivalences.  As a consequence,
$$
K_{0}(\cS(\Gamma, \mu)) = \Z\{[p_{\alpha}] | \alpha \in V\} \text{ and } K_{1}(\cS(\Gamma, \mu)) = \{0\}
$$
where the first group is the free abelian group on the vertices of $\Gamma$.
\end{cor}
In particular, this shows that the $KK$-groups are independent of the weighting $\mu$.  The positive cone $K_{0}^{+}(\cS(\Gamma, \mu))$ will depend on $\mu$ however.  See Subsection $\ref{sec:ideal}$ below.


\subsection{Ideal Structure}\label{sec:ideal}

In this section we fix a finite, connected graph, $(\Gamma, \mu)$ with at least two undirected edges.  We will give an explicit description of the closed ideals of $\cS(\Gamma, \mu)$.  To begin,  we will define two sets $V_{>}(\Gamma, \mu)$ and $V_{=}(\Gamma, \mu)$ as follows
$$
V_{>}(\Gamma, \mu) = \set{\alpha \in V}{\mu(\alpha) > \sum_{\beta \sim \alpha} n_{\alpha, \beta} \mu(\beta)} \text{ and } V_{=}(\Gamma, \mu) = \set{\alpha \in V}{\mu(\alpha) = \sum_{\beta \sim \alpha} n_{\alpha, \beta} \mu(\beta)} 
$$  
where we write $\beta \sim \alpha$ to indicate that $\beta$ and $\alpha$ are the endpoints of at least one edge of $\Gamma$, and we let $n_{\alpha, \beta}$ be the number of unoriented edges that have $\alpha$ and $\beta$ as endpoints.  When $\Gamma$ and $\mu$ are understood, we will write $V_{>}$ instead of $V_{>}(\Gamma, \mu)$, and $V_{=}$ for $V_{=}(\Gamma, \mu)$.   We will also write $V_{\geq}$ to denote the following set
$$
V_{\geq} = V_{=} \cup V_{>}
$$

We will show that there is an aesthetically pleasing description of the ideal structure of $\cS(\Gamma, \mu)$ in terms of these vertex sets.  To begin, we will state some results on the simplicity and stable rank of certain reduced (amalgamated) free products.

\begin{nota}\label{nota:free}
\mbox{}

\begin{enumerate}
\item Let $A_{1}$ and $A_{2}$ be two unital C$^{*}$-algebras with faithful tracial states $\tr_{1}$ and $\tr_{2}$ respectively.  We will write $A_{1} * A_{2}$ as the \emph{reduced} free product with respect to the tracial states $\tr_{1}$ and $\tr_{2}$. 

\item More generally, if $A_{1}$ and $A_{2}$ are two C$^{*}$-algebras with faithful tracial states $\tr_{1}$ and $\tr_{2}$ respectively and with trace preserving conditional expectations $E_{i}:  A_{i} \rightarrow D$ onto a common subalgebra $D$, then $A_{1} \underset{D}{*} A_{2}$ will denote as the \emph{reduced} amalgamated free product with respect to the expectations $E_{1}$ and $E_{2}$.

\item If $A$ is a unital C$^{*}$ algebra with a finite faithful (not necessarily normalized) trace $\tr$, then we write
$$
A = \underset{\mu_{1}}{\overset{p_{1}}{B_{1}}} \oplus \cdots \oplus \underset{\mu_{n}}{\overset{p_{n}}{B_{n}}}
$$
If $A = B_{1} \oplus \cdots \oplus B_{n}$ for unital C$^*$-algebras $B_{1}, \dots, B_{n}$ with identities $p_{1}, \dots, p_{n}$ respectively.  We have $\tr(p_{k}) = \mu_{k}$ for each $k$.
\end{enumerate}
\end{nota}

With this notation in hand, we will now state several free product results that will be used in our analysis of $\cS(\Gamma, \mu)$.

\begin{lem}[\cite{MR1473439}]\label{lem:sumproj}
Let $A = A_{1} * A_{2}$ and let $p \in A_{1}$, $q \in A_{2}$ be projections.
\be

\item If $\tr(p + q) > 1$, then $p + q$ is invertible.

\item If $\tr(p + q) < 1$, then $\{0\}$ is an isolated point in the spectrum of $p + q$ and if $\mu$ is the law of $p + q$ according to $\tr$, then $\mu\{0\} = 1 - \tr(p + q)$.

\item If $\tr(p + q) = 1$, the support projection of $p + q$ is 1; however, $p + q$ is not invertible, and $\{0\}$ is not isolated in the spectrum of $p + q$.

\ee
\end{lem}

A positivity argument and induction lead to the following corollary
\begin{cor}\label{cor:sumproj}
Let $n \geq 2$, $A = A_{1} * A_{2} * \cdots * A_{n}$ and $p_{i} \in A_{i}$ be a projection for each $i$.
\be

\item If $\sum_{i=1}^{n}\tr(p_{i}) > 1$, then $\sum_{i=1}^{n} p_{n}$ is invertible.

\item If $\sum_{i=1}^{n}\tr(p_{i}) < 1$ then $\{0\}$ is an isolated point in the spectrum of $\sum_{i=1}^{n} p_{n}$ and if $\mu$ is the law of $p + q$ according to $\tr$, then $\mu\{0\} = 1 - \tr(p + q)$.  

\item If $\sum_{i=1}^{n}\tr(p_{i}) = 1$, the support projection of $\sum_{i=1}^{n} p_{n}$ is 1; however, $\sum_{i=1}^{n} p_{n}$ is not invertible, and $\{0\}$ is not isolated in the spectrum of $p + q$

\ee

\end{cor}

\begin{defn} \label{defn:stable}\mbox{}
\be
\item If $A$ is a C$^{*}$-algebra, then $K_{0}(A)^{+} = \set{x \in K_{0}(A)}{x = [p] \text{ for some  projection } p \in M_{n}(A)}$

\item If $A$ is a unital C$^{*}$-algebra, we say that $A$ has stable rank one if the set of invertible elements in $A$ is norm-dense in $A$.  If $A$ is nonunital, $A$ is said to have stable rank one if its unitization does.

\ee
\end{defn}

\begin{thm}\label{thm:Avi}
\mbox{}
Assume $A_{1}$ and $A_{2}$ are unital C$^{*}$-algebras with faithful tracial states $\tr_{1}$ and $\tr_{2}$ respectively.  Suppose that $A_{1}$ contains a unitary $u$, and that $A_{2}$ contains unitary elements $v$ and $w$ satisfying
$$
\tr_{1}(u) = \tr_{2}(v) = \tr_{2}(w) = \tr_{2}(vw^{*}) = 0.
$$ 
Let $A = A_{1} * A_{2}$.
\be
\item \cite{MR654842} $A$ is simple with unique tracial state.

\item \cite{MR1478545}  $A$ has stable rank one.  

\item \cite{MR1601917} Let $j_{i}: A_{i} \rightarrow A$ be the canonical inclusions, and let  $j_{i}^{*}: K_{0}(A_{i}) \rightarrow K_{0}(A)$ be the induced map on $K_{0}$.  If $G$ is the subgroup of $K_{0}(A)$ generated by $j_{1}^{*}(A_{1})$ and $j_{2}^{*}(A_{2})$, then 
$$
G \cap K_{0}(A)^{+} = \set{x \in G}{\tr(x) > 0} \cup \{0\}.
$$ 
\ee
\end{thm}

\begin{defn}\label{defn:diffuse}
Let $A$ be a unital, abelian C$^{*}$ algebra with faithful state $\phi$.  Identify $A \cong C(X)$ with $X$ a compact Hausdorff space, and note that $\phi(f) = \int_{X}f(x)d\mu(x)$ for some probability measure $\mu$ on $X$.  We say that $(A, \phi)$ is \emph{diffuse} if $\mu$ has no atoms.  This is equivalent to $A$ containing a \emph{Haar unitary}, i.e. a unitary element $u$ satisfying $\phi(u^{k}) = 0$ for all $k \in \Z \setminus \{0\}$ \cite{MR1478545}.

\end{defn}

\begin{thm}[\cite{MR1473439}] \label{thm:diffuse}
\mbox{}
\be
\item 
Suppose $A =  A_{1} * A_{2}$ where $A_{1}$ contains a diffuse abelian C$^{*}$-subalgebra and $A_{2} \neq \C$.  
Then $A$ is simple, has stable rank one, and has unique tracial state $\tr$.  

\item 
Suppose $A = A_{1} * \cdots * A_{n}$, where each $A_{i}$ has the form
$$
A_{i} = \underset{\mu_{i}}{\cD} \oplus \C
$$
where $\cD$ contains a diffuse abelian subalgebra.  
Then $A$ is simple and has unique trace if and only if $\sum_{i=1}^{n}\mu_{i} > 1$.  
$A$ always has stable rank one, regardless of the weighting.
\ee
\end{thm}

\begin{lem}[\cite{MR1201693,MR1473439,MR3266249}] \label{lem:compressfree}
\mbox{}
\be
\item Suppose that $A$, $B$, and $C$ are tracial C$^{*}$-algebras with
$$
\mathcal{D} = (\overset{p}{A} \oplus B) * C,
$$
and $\mathcal{D}$ is endowed with the canonical free product trace.  Then
$
p\mathcal{D}p = A * p\left((\overset{p}{\C} \oplus B) * C\right)p.
$

\item Suppose there are two unital, tracial, C$^{*}$-algebras $\overset{p + r}{B} \oplus \overset{r'}{\C} \text{ and } C$ which both contain $D = \overset{p}{\C} \oplus \overset{q}{\C}$ as a unital C$^{*}$-subalgebra with $q = r + r'$. Assume that the algebras are equipped with conditional expectations $E^{1}_{D}$ and $E^{2}_{D}$ onto $D$ respectively as well as traces $\tr_{1}$ and $\tr_{2}$ so that $\tr_{i} = \tr_{i} \circ E^{i}_{D}$ for $i = 1, 2$, and the restrictions of $\tr_{1}$ and $\tr_{2}$ to $D$ coincide.  Form the reduced amalgamated free product
$$
\mathcal{D} = \left( \overset{p + r}{B} \oplus \overset{r'}{\C} \right) \underset{D}{*}\, C.
$$
Then
$$
(p+r)\mathcal{D}(p+r) = (B, E^{1}_{D'}) \underset{D'}{*}\, \left((p+r)\left(\left(\overset{p}{\C} \oplus \overset{r}{\C} \oplus \overset{r'}{\C}\right) \underset{D}{*}\, C\, , E^{2}_{D}\right)(p+r), E^{2}_{D'}\right)
$$
where $D' = \overset{p}{\C} \oplus \overset{r}{\C}$, the conditional expectations $E^{i}_{D'}$ onto $D'$ are the trace preserving ones, and the free product is reduced.

\ee
\end{lem}

\begin{lem}[\cite{MR2782689}]\label{lem:simpleamalgamated}
Let $C_{1}$ and $C_{2}$ be unital C$^{*}$ algebras containing the unital C$^{*}$ subalgebra $D$ unitally.  Suppose each $C_{i}$ is equipped with a trace $\tr_{i}$ such that $\tr_{1}$ and $\tr_{2}$ coincide on $D$ and that there exist trace preserving conditional expectations $E^{i}_{D}$ of $C_{i}$ onto $D$.  Consider the reduced amalgamated free product with conditional expectation $E$
$$
(C,E) = (C_{1}, E^{1}_{D}) \underset{D}{\Asterisk} (C_{2}, E^{2}_{D}).
$$
Then $C$ is simple and has unique trace $\tr = \tr_{1} \circ E_{D} = \tr_{2} \circ E_{D}$ provided the following conditions hold:
\be
\item
There exist unitaries $u_{1} \in C_{1}$ and $u_{2}, u_{2}'$ in $C_{2}$ such that $E_{D}(u_{1}) = 0 = E_{D}(u_{2}) = E_{D}(u_{2}') = E(u_{2}^{*}u_{2}')$.
\item
For every $a_{1}, ..., a_{n} \in D$ with zero trace, there exists a unitary $u \in C_{2}$ with expectation 0 such that $E_{D}(ua_{i}u^{*}) = 0$ for each $i$.
\item
There are unitaries $w, v \in C_{2}$ with expectation 0 such that $E_{D}(wav) = 0$ for all $a \in D$.
\ee
Furthermore, let $G$ be the subgroup of $K_{0}(C)$ which is generated by $j_{1}^{*}(K_{0}(C_{1}))$ and $j_{2}^{*}(K_{0}(C_{2}))$ with $j_{i}: C_{i} \rightarrow C$ the canonical inclusion.  If the above three conditions hold, then
$$
K_{0}(C)^{+} \cap G = \{x \in G : \tr(x) > 0\} \cup \{0\}.
$$
\end{lem}

\begin{lem}[\cite{MR3266249}]\label{lem:amalgamsr1}
Suppose that $B_{1}$ and $B_{2}$ are unital separable C$^{*}$ algebras both unitally containing $D = \C^{2}$ as a subalgebra.  Assume that $B_{1}$ and $B_{2}$ are equipped with faithful traces $\tr_{1}$ and $\tr_{2}$ respectively such that $\tr_{1}$ and $\tr_{2}$ agree on $D$.  Let $E^{i}_{D}$ be the trace preserving conditional expectation from $B_{i}$ to $D$ for $i \in \{1, 2\}$.  Form the reduced amalgamated free product
$$
(B, E) = (B_{1}, E^{1}_{D}) \underset{D}{*} (B_{2}, E^{2}_{D}).
$$
Let $p$ and $q$ be the two minimal projections in $D$, and assume $\tr_{i}(p) = \tr_{i}(q)$.  Suppose $B_{1}$ contains a unitary $u_{1}$ and $B_{2}$ contains unitaries $u_{2}$ and $u_{2}'$ with $v= pvp + qvq$ for $v \in \{u_{1}, u_{2}, u_{2}'\}$ which also satisfy
$$
E(u_{1}) = 0 = E(u_{2}) = E(u_{2}') = E(u_{2}^{*}u_{2}').
$$
Then $pB p$ and $qB q$ both have stable rank 1.
\end{lem}

The key to unlocking the ideal structure of $\cS(\Gamma, \mu)$ is showing the existence of a minimal ideal that ``avoids" the sets $V_{>}(\Gamma, \mu)$ and $V_{=}(\Gamma, \mu)$.  To begin, we fix a vertex, $\alpha$ of minimal weight in $V$.  If there are two such vertices which are joined by an edge, we will choose $\alpha$ to have the additional property that there is another edge connected to $\alpha$.  Notice that this choice guarantees that $\alpha \not\in V_{\geq}(\Gamma, \mu)$.

\begin{lem}\label{lem:Iminimal}
Let $I$ be the closed ideal in $\cS(\Gamma, \mu)$ generated by $p_{\alpha}$.  Then $I$ contains $\set{p_{\beta}}{\beta \not\in V_{\geq}(\Gamma, \mu)}$.  Furthermore, $\cS(\Gamma, \mu)/I \cong \bigoplus_{\gamma \in V_{\geq 0}} \C$.  In addition, $I$ is unital if and only if $V_{=}(\Gamma, \mu)$ is empty.
\end{lem}

\begin{proof}
Let $\beta$ be a neighboring vertex of $\alpha$.\\

\underline{Case 1}:  Suppose that $\beta \not\in V_{\geq}(\Gamma, \mu)$.  Let $e_{1}, \dots, e_{n}$ denote the edges in $\Gamma$ which have $\beta$ as an endpoint, and let $\e_{1}, \dots, \e_{n}$ be the corresponding oriented edges in $\vec{\Gamma}$ having $t(\e_{i}) = \beta$ for all $i$, and $s(\e_{1}) = \alpha$.  Using Theorem \ref{thm:freeexpectation}, the traces of the support projections of the elements $X_{\e_{i}}^{*}X_{\e_{i}}$ add up to $\sum_{\gamma \sim \beta} n_{\gamma, \beta}\mu(\beta)$ which by hypothesis exceeds $\mu(\beta)$.  It follows from Theorem \ref{thm:diffuse} that the C$^{*}$-algebra generated by the elements $X_{\e_{i}}^{*}X_{\e_{i}}$ is simple which means that $p_{\beta}$ is in the ideal generated by these elements.  Noting that $X_{\e_{1}}^{*}X_{\e_{1}} = X_{\e_{1}}^{*}p_{\alpha}X_{\e_{1}}$, this shows that $p_{\beta} \in I$.

\underline{Case 2}: Suppose that $\beta \in V_{\geq}$ and let $e_{1}, \dots, e_{n}$ and $\epsilon_{1}, \dots, \epsilon_{n}$ as in the previous part.  Every element in $I$ which is supported under $p_{\beta}$ must be a norm limit of elements of the form
$$
\sum_{i, j} X^{*}_{\e_i}x_{i,j}p_{\alpha}y_{i,j}X_{\e_j}
$$ 
for $x_{i, j}, y_{i, j}$ in the $*$-algebra generated by $\set{X_{\e}}{\e \in \vec{E}}$.

If $n \geq 2$, then the support projections $p_{\e_{i}}$ of $X_{\e_{i}}^{*}X_{\e_{i}}$ are in $\cS(\Gamma, \mu)$ and are free with respect to $\Tr(p_{\alpha} \cdot p_{\alpha})$.  We have $\sum_{i = 1}^{n} \Tr(p_{\e_{i}}) \leq \Tr(p_{\beta})$.  Furthermore, if $x \in p_{\beta}Ip_{\beta}$, then $x$ must be in the hereditary C$^*$-algebra $(\sum_{i  =1}^{n} p_{\e_{i}})\cS(\Gamma)(\sum_{i  =1}^{n} p_{\e_{i}})$.   By Corollary \ref{cor:sumproj}, $\sum_{i  =1}^{n} p_{\e_{i}}$ is not invertible in $p_{\beta}\cS(\Gamma)p_{\beta}$.  This implies $p_{\beta} \not\in I$.  The unital C$^{*}$-subalgebra in $p_{\beta}\cS(\Gamma, \mu)p_{\beta}$ generated by the elements $X_{\e_{i}}^{*}X_{\e_{i}}$ is the reduced free product
$$
B = \underset{i=1}{\overset{n}{\Asterisk}} (\overset{p_{\e_{i}}}{C^{*}(X_{\e_{i}}^{*}X_{\e_{i}})} \oplus \C).
$$
By Lemma \ref{lem:compressfree} and Theorem \ref{thm:diffuse}, $p_{\e_{i}}Bp_{\e_{i}}$ is simple for each $i$ and the element $p_{\e_{i}}X_{\e_{1}}^{*}X_{\e_{1}}p_{\e_{i}}$ is nonzero.  This means that $p_{\e_{i}} \in I$ for all $i$, hence implying that $p_{s(\e_{i})} \in I$ for all $i$.

If $n = 1$, then every element in $I$ supported under $p_{\alpha}$ is a norm limit of elements of the form $X_{\e_{1}}^{*}yX_{\e_{1}}$.  Since $X_{\e_{1}}^{*}X_{\e_{1}}$ is not invertible in $p_{\beta}\cS(\Gamma, \mu)p_{\beta}$ it follows that $p_{\beta}\not\in I$.

Continuing this argument inductively, we see that $I$ contains $\set{p_{\gamma}}{\gamma \in V\setminus V_{\geq}}$ and does not intersect $V_{\geq}$.  This means that $I$ is generated as an ideal by $\set{X_{\e}}{\e \in \vec{E}}$ since at least one vertex in $V\setminus V_{\geq}$ is an endpoint of $\e$ for each $\e \in E$.  If $\gamma \in V_{\geq}$ then every element in $p_{\gamma}\cS(\Gamma, \mu)p_{\gamma}$ is a norm limit of elements of the form $cp_{\gamma} + p_{\gamma}x p_{\gamma}$ where $x$ is a polynomial in the $X_{\e}$'s.  The arguments in case 2 show that $p_{\gamma}$ is not a norm limit of expressions  of the form $p_{\gamma}x p_{\gamma}$.  This implies that $\cS(\Gamma, \mu)/I \cong \bigoplus_{\gamma \in V_{\geq}} \C$.

Finally we verify the statement on the unitality of $I$. Let $\gamma \in V_{\geq}$ and $\e'_{1}, \dots, \e'_{m}$ be all of the edges satisfying $t(\e'_{i}) = \gamma$.  If $V_{=}$ is empty and $p_{\e'_{i}}$ is the support projection of $X_{\e_{i}}^{*}X_{\e_{i}}$, then since $\sum_{i=1}^{m} \Tr(p_{\e'_{i}}) < \Tr(p_\gamma)$, $\{0\}$ is open in the spectrum of $\sum_{i} p_{i}$ from Corollary \ref{cor:sumproj} which implies that the support projection, $q_{\gamma}$ of $\sum_{i=1} p_{\e'_i}$ is in $I$.  
The element $\sum_{\gamma \not\in V_{\geq}} p_{\gamma} + \sum_{\gamma \in V_{\geq}} q_{\gamma}$ is the unit for $I$. 

If $V_{=}$ is not empty, $m \geq 2$, and $\gamma \in V_{=}$, then since $\sum_{i} \Tr(p_{\e_{i}}) = \Tr(p_\gamma)$, $p_{\gamma}$ is in the strong closure (in $L^{2}(\cS(\Gamma, \mu), \Tr)$) of the C$^{*}$-algebra generated by $\sum_{i} \Tr(p_{\e_{i}})$ \cite{MR3110503}, so it follows that $p_{\gamma}$ is the support projection of $\sum_{i} \Tr(p_{\e_{i}})$.  Therefore $I$ is not unital.  Finally, if $m = 1$ and $\gamma \in V_{=}$, then since the law of $X_{\e'_{1}}^{*}X_{\e'_{1}}$ in $p_{\gamma}(\cS(\Gamma), \Tr)p_{\gamma}$ has no atoms, it follows that $p_{\gamma}$ is in the strong closure of $I$ so $I$ is not unital. 
\end{proof}

\begin{cor} If $V_{=}$ is empty, then 
$$
\cS(\Gamma, \mu) = I \oplus \bigoplus_{\gamma \in V_{>}} \overset{r_{\gamma}}{\C}
$$
with $r_{\gamma} \leq p_{\gamma}$ and $\Tr(r_{\gamma}) = \mu(\gamma) - \sum_{\beta \sim \gamma}n_{\alpha, \beta}\mu(\beta)$.  If $V_{=}$ is not empty, then 
$$
\cS(\Gamma, \mu) = \mathcal{I} \oplus \bigoplus_{\gamma \in V_{>}} \overset{r_{\gamma}}{\C}
$$
where $\mathcal{I}$ is unital, the strong operator closures of $I$ and $\mathcal{I}$ coincide in $L^{2}(\cS(\Gamma, \mu), \Tr)$, and $\mathcal{I}/I \cong \bigoplus_{\beta \in V_{=}} \C$.

\end{cor}

We will now show that $I$ is minimal.

\begin{lem}\label{lem:Isimple}
Let $I$ be as in Lemma \ref{lem:Iminimal}.  Then $I$ is simple, has unique trace, and has stable rank 1.
\end{lem}

\begin{proof} The arguments here mirror arguments used in Section 4 of \cite{MR3266249}.\\

\underline{Case 1}:  We first assume that there is a vertex $\alpha$ of minimal weight such that $\Gamma = \Gamma_{1} \cup \Gamma_{2}$, where $\Gamma_{1}$ and $\Gamma_{2}$ are two connected subgraphs of $\Gamma$, each of which have a nonempty edge set, share no edges in common, and intersect only at the vertex $\alpha$.   
We have $p_{\alpha}\cS(\Gamma,\mu)p_{\alpha} = p_{\alpha}\cS(\Gamma_{1}, \mu)p_{\alpha} * p_{\alpha}\cS(\Gamma_{2}, \mu)p_{\alpha}$.  Let $\e_{1}$ and $\e_{2}$ be oriented edges with $t(\e_i) = \alpha$ and $\e_{i} \in \vec{E}(\Gamma_{i})$.  
The elements $X_{\e_i}^{*}X_{\e_i}$ generate diffuse abelian C$^{*}$-subalgebras of $p_{\alpha}\cS(\Gamma,\mu)p_{\alpha}$, so it follows from Theorem \ref{thm:diffuse} that $p_{\alpha}\cS(\Gamma,\mu)p_{\alpha}$ is simple, has unique trace, and has stable rank 1 from Theorem \ref{thm:Avi} or Theorem \ref{thm:diffuse}.  This implies that $I$ is simple, has unique trace, and has stable rank 1. 
 
 \underline{Case 2}: We assume that $\alpha$ is only connected to one other vertex $\beta$ by only one edge $e_{1}$.  Let $\e_{1} \in \vec{E}$ have $s(\e_{1}) = \alpha$ and $t(\e_{1}) = \beta$.  Notice that the assumption on $\alpha$ implies that there is at least one other undirected edge $e_{2}$ with $\beta$ as an endpoint, and $\mu(\beta) > \mu(\alpha)$ since equality implies that we are in Case 1 above.  Assume that $t(\e_{2}) = \beta$.   Let $\tilde{\Gamma}$ be the graph which is obtained by removing the edge $e_{1}$ from $\Gamma$.  Let $q_{\alpha}$ be the support projection of $X_{\e_{1}}^{*}X_{\e_{1}}$ and note that $q_{\alpha} \leq p_{\beta}$ and $q_{\alpha}$ is equivalent to $p_{\alpha}$.  If $B$ is the C$^{*}$-algebra generated $q_{\beta}$ and $\cS(\tilde{\Gamma})$   It follows from Lemma \ref{lem:compressfree} that   
 $$
 q_{\alpha}\cS(\Gamma)q_{\alpha} = q_{\alpha}C^{*}(X_{\e}^{*}X_{\e})q_{\alpha} * q_{\alpha}Bq_{\alpha}.
 $$
 The C$^{*}$-algebra $q_{\alpha}C^{*}(X_{\e}^{*}X_{\e})q_{\alpha}$ is diffuse and abelian.  The algebra $q_{\alpha}Bq_{\alpha}$ contains the element $q_{\alpha}X_{\e_{2}}^{*}X_{\e_{2}}q_{\alpha}$ which generates a diffuse von-Neumann algebra since the support projection of $X_{\e_{2}}^{*}X_{\e_{2}}$ has trace at least as large as $\mu(\alpha)$.  It follows that $q_{\alpha}Bq_{\alpha}$ has a unitary of trace zero.  Therefore, from Theorem \ref{thm:Avi}, $q_{\alpha}\cS(\Gamma)q_{\alpha}$ is simple, has unique trace, and has stable rank 1.  Therefore, $I$ is simple, has unique trace, and has stable rank 1.
 
 \underline{Case 3}: We assume that there are two distinct vertices, $\alpha$ and $\beta$, an edge $e_{1}$ joining $\alpha$ and $\beta$, and a path from $\alpha$ to $\beta$ that avoids $e_{1}$.  We may also assume that $\alpha$ is of minimal weight, and $\alpha \not\in V_{\geq}$.  Let $\e_{1} \in E(\vec{\Gamma})$ be the edge associated to $e_{1}$ which satisfies $t(\e_{1}) = \beta$ and $s(\e_{1}) = \alpha$, and set $q_{\alpha}$ to be the support projection of $X_{\e_{1}}^{*}X_{\e_{1}}$ and note that $q_{\alpha} \leq p_{\beta}$.  Also note that $q_{\alpha} \in \cS(\Gamma, \mu)$ since if $X_{\e_{1}}^{*}X_{\e_{1}}$ has connected spectrum, then $q_{\alpha} = p_{\beta}$.  
 
Set $\cB = (p_{\alpha} + q_{\alpha})\cS(\Gamma, \mu)(p_{\alpha} + q_{\alpha})$.  Let $\Gamma'$ be the subgraph of $\Gamma$ obtained by deleting the edge $e_{1}$, and and $\Gamma''$ be the subgraph of $\Gamma$ whose vertices are $\alpha$ and $\beta$ and whose edge set is $\{e_{1}\}$.  Set $D = C^{*}(\{ p_{\alpha}, p_{\beta} \})$  and $D' = C^{*}(\{p_{\alpha}, q_{\alpha}\})$.  From Lemma \ref{lem:compressfree}, it follows that
$$
\cB =  (p_{\alpha} + q_{\alpha})\cS(\Gamma'', \mu)(p_{\alpha} + q_{\alpha}) \underset{D'}{*} (p_{\alpha} + q_{\alpha})\cB'(p_{\alpha} + q_{\alpha})
$$
where the conditional expectations are the trace preserving ones, and 
$$
\cB' = \left(\overset{p_{\alpha}}{\C} \oplus \overset{q_{\alpha}}{\C} \oplus \overset{p_{\beta} - q_{\alpha}}{\C}\right) \underset{D}{*} (p_{\alpha} + p_{\beta})\cS(\Gamma', \mu)(p_{\alpha} + q_{\beta}). 
$$
We now show that $\cB$ satisfies the conditions in Lemma \ref{lem:simpleamalgamated}.

Theorem \ref{thm:freeexpectation} determines the structure of $(p_{\alpha} + q_{\alpha})\cS(\Gamma'', \mu)(p_{\alpha} + q_{\alpha})$.  Explicitly, if $\mu(\alpha) < \mu(\beta)$, then 
$$
(p_{\alpha} + q_{\alpha})\cS(\Gamma'', \mu)(p_{\alpha} + q_{\alpha}) \cong M_{2}(\C) \otimes C[0, 1]
$$ 
with trace $\Tr_{M_{2}}(\C) \otimes \int\cdot d\lambda$ with $d\lambda$ Lebesgue measure.  In this isomorphism, we have
$$
p_{\alpha} \mapsto 
\begin{pmatrix}
1 & 0\\
0 & 0
\end{pmatrix}
\text{ and }
q_{\alpha} \mapsto
\begin{pmatrix}
0 & 0\\
0 & 1
\end{pmatrix}.
$$
If $\mu(\alpha) = \mu(\beta)$, then 
$$
(p_{\alpha} + q_{\alpha})\cS(\Gamma'', \mu)(p_{\alpha} + q_{\alpha}) \cong\set{f: [0, 1] \rightarrow M_{2}(\C)}{f \text{ is continuous and } f(0) \text{ is 
diagonal}}
$$
with the above trace and identifications for $p_{\alpha}$ and $q_{\alpha}$.  In either case, the unitary
$$
U =  \begin{pmatrix}
\cos(2\pi t) & -\sin(2\pi t)\\
\sin(2\pi t) & \cos(2\pi t)
\end{pmatrix}
$$
lies in $(p_{\alpha} + q_{\alpha})\cS(\Gamma'', \mu)(p_{\alpha} + q_{\alpha})$.  The traceless elements of $D'$ are spanned by 
$$
x = \begin{pmatrix}
1 & 0\\
0 & -1
\end{pmatrix}
$$
and it is easy to check that $UxU^{*}$ has zero expectation so condition (2) from Lemma \ref{lem:simpleamalgamated} is satisfied.  Note that one has
$$
E(u) = E(u^{*}) = E(u^{2}) = E(u(u^{*})^{*}) = 0.
$$
Choose edges $e_{2}$ and $e_{3}$ of $\Gamma'$ whose oriented versions satisfy $t(\e_{2}) = \alpha$ and $t(\e_{3}) = \beta$.  There is a Haar unitary $v_{1}$ in $p_{\alpha}\cS(\Gamma', \mu)p_{\alpha}$ in the continuous functional calculus of $X_{\e_{2}}^{*}X_{\e_{2}}$.  As in the previous case, $q_{\alpha}X_{\e_{3}}^{*}X_{\e_{3}}q_{\alpha}$ generates a diffuse von Neumann algebra, so it follows that there is a unitary, $v_{2}$, of trace zero in $q_{\alpha}\cS(\Gamma', \mu)q_{\alpha}$.  This implies that $v_{1} + v_{2}$ is an expectationless unitary in $(p_{\alpha} + q_{\alpha})\cB'(p_{\alpha} + q_{\alpha})$ so condition (1) from Lemma \ref{lem:simpleamalgamated} is satisfied.

Finally, under the matrix algebra identification of $(p_{\alpha} + q_{\alpha})\cS(\Gamma'', \mu)(p_{\alpha} + q_{\alpha})$, set 
$$
V = \begin{pmatrix} 
e^{2i\pi t} & 0 \\
0 & e^{2i\pi t}
\end{pmatrix}.
$$
It is easy to see that $V$ is expectationless, and that $E(VyV) = 0$ for all $y \in D'$.  This implies that condition (3) from Lemma \ref{lem:simpleamalgamated} is satisfied.  This implies $\cB$ is simple and has unique trace, hence $I$ is simple and has unique trace. 

The verification of this case shows that the conditions in Lemma \ref{lem:amalgamsr1} are satisfied, so $p_{\alpha}\cB p_{\alpha}$ has stable rank 1.  This implies $I$ has stable rank 1 as well.
\end{proof}

There are several immediate corollaries

\begin{cor}
$\cS(\Gamma, \mu)$ is simple if and only if $V_{\geq}$ is empty.
\end{cor}

\begin{cor}
For each $\beta \in V_{\geq}$, the ideal $I_{\beta}$, generated by $\{p_{\gamma}: \gamma \neq \beta\}$ is maximal and of co-dimension 1.  Furthermore, every ideal in $\cS(\Gamma, \mu)$ is an intersection of the ideals $I_{\beta}$, and 
$$
I = \bigcap_{\beta \in V_{\geq}}I_{\beta}
$$ 
\end{cor}

\begin{cor} 
$\cS(\Gamma, \mu)$ has stable rank 1.
\end{cor}

\begin{proof}
$\cS(\Gamma, \mu)/I$ is finite-dimensional.  The result follows from \cite{MR693043}.
\end{proof}

We now turn our attention to the $K$-groups and positive cone of $I$.

\begin{lem}\label{lem:KI}
The $K$-groups of $I$ are as follows:
$$
K_{0}(I) = \Z\set{[p_{\beta}]}{\beta \in V\setminus V_{\geq}} \text{ and } K_{1}(I) = \{0\}
$$
where $\Z\set{[p_{\beta}]}{\beta \in V\setminus V_{\geq}}$ is the free abelian group on the vertices of $\Gamma$ which are not in $V \setminus V_{\geq}$.
\end{lem}
\begin{proof}
Consider the six term exact sequence\\
$$
\begin{tikzpicture}
	\node at (-5, 0) {$K_{0}(I)$};
	\node at (0, 0) {$K_{0}(\cS(\Gamma))$};
	\node at (5, 0) {$K_{0}(\cS(\Gamma)/I)$};
	\node at (5, -4) {$K_{1}(I)$};
	\node at (0, -4) {$K_{1}(\cS(\Gamma))$};
	\node at (-5, -4) {$K_{1}(\cS(\Gamma)/I)$};
	\draw [->] (-4.2, 0) -- (-1, 0);
	\node at (-2.6, .3) {$\iota^{0}_{*}$};
	\draw [->] (1, 0) -- (3.7, 0);
	\node at (2.35, .3) {$\pi^{0}_{*}$};
	\draw [->] (5, -.5) -- (5, -3.5);
	\node at (5.3, -2) {$\partial_{0}$};
	\draw [->] (4.2, -4) -- (1, -4);
	\node at (2.6, -4.3) {$\iota^{1}_{*}$};
	\draw [->] (-1, -4) -- (-3.7, -4);
	\node at (-2.35, -4.3) {$\pi^{1}_{*}$};
	\draw [->] (-5, -3.5) -- (-5, -.5);
	\node at (-5.3, -2) {$\partial_{1}$};
\end{tikzpicture}
$$
where $\iota_{*}^{i}$ are the induced maps from the canonical inclusion $\iota: I \rightarrow \cS(\Gamma)$, $\pi^{i}_{*}$ are the induced maps from the canonical quotient map $\pi: \cS(\Gamma) \rightarrow \cS(\Gamma)/I$, and $\partial_{i}$ are the connecting maps.   Recall that 
$$
\cS(\Gamma)/I \cong \bigoplus_{\beta \in V_{\geq}} \C
$$
and the induced map C$^{*}\set{p_{\beta}}{\beta \in V_{\geq}} \rightarrow \cS(\Gamma)/I$ is an isomorphism.  Therefore, $\pi^{0}_{*}$ is surjective, which implies that $\partial_{0}$ is the zero map, which implies $\iota^{1}_{*}$ is injective.  Since $K_{1}(\cS(\Gamma)) = \{0\}$ it follows that $K_{1}(I) = \{0\}$.

The image of $\iota_{*}^{0}$ is $\Z\set{p_{\beta}}{\beta \in V\setminus V_{\geq}}$.  As $\cS(\Gamma)/I$ is finite-dimensional, $K_{1}(\cS(\Gamma)/I) = \{0\}$ which implies $\iota^{0}_{*}$, hence
$$
K_{0}(I) = \Z\set{[p_{\beta}]}{\beta \in V\setminus V_{\geq}}.
$$ 
\end{proof}

\begin{lem} $K_{0}(I)^{+} = \set{x \in K_{0}(I)}{\Tr(x) > 0} \cup \{0\}$.

\end{lem}

\begin{proof}  The proof of Lemma \ref{lem:Isimple} shows that in the three cases where a hereditary subalgebra of $I$ was expressed as a(n) (amalgamated) free product of $A$ and $B$, then the subgroup, $G \subset K_{0}(I)$,  generated by $\iota^{0}_{*}(A)$ and $\iota^{0}_{*}(B)$ has the desired structure of the positive cone.  We simply need to show that $G = K_{0}(I)$.  This will be done by analyzing the three free product cases.

\underline{Case 1}:  If $\Gamma = \Gamma_{1} \cup \Gamma_{2}$ as in the proof of Lemma \ref{lem:Isimple}, then from the description of $\Gamma$, if $\beta \neq \alpha$ and $\beta \in V(\Gamma) \setminus V_{\geq}(\Gamma)$, then $\beta \in  V(\Gamma_{i}) \setminus V_{\geq}(\Gamma_{i})$ for exactly one $i \in \{1, 2\}$.  From the proof of Lemma \ref{lem:Iminimal}, it follows that $p_{\beta}$ is in the ideal generated by $p_{\alpha}$ in $\cS(\Gamma_{i})$.  It follows $K_{0}(I)$ is generated by $K_{0}(p_{\alpha}\cS(\Gamma_{j})p_{\alpha})$ for $j \in \{1, 2\}$.

\underline{Case 2}:  Assume that $\Gamma$ satisfies the conditions in Case 2 in Lemma \ref{lem:Isimple} and all notation here comes from Case 2 in the proof of Lemma \ref{lem:Isimple}.  Let $\e_{2}, \dots, \e_{n}$ be all of the edges distinct from $\e_{1}$ which have target $\beta$.  Let $p_{\e_i}$ be the support projection of $X_{\e_{i}}^{*}X_{\e_{i}}$ for $i \in \{1, \dots , n\}$.  If $i > 1$, we see that $p_{\e_i}Bp_{\e_i}$ contains the algebra
$$
p_{\e_i}((\overset{q_{\alpha}}{\C} \oplus \C) * (\overset{p_{\e_i}}{C^{*}(X_{\e_{i}}^{*}X_{\e_{i}})} \oplus \overset{p_{\beta} - p_{\e_i}}{\C}))p_{\e_i} = C^{*}(X_{\e_{i}}^{*}X_{\e_{i}}) * p_{\e_i}((\overset{q_{\alpha}}{\C} \oplus \C) * (\overset{p_{\e_i}}{\C} \oplus \overset{p_{\beta} - p_{\e_i}}{\C}))p_{\e_i}
$$
which is simple since C$^{*}(X_{\e_{i}}^{*}X_{\e_{i}})$ is diffuse.  This implies that $p_{\e_i}$ is in the ideal generated by $p_{\e_1}$ in $B$.  

If $\beta \in V\setminus V_{\geq}$, then $\sum_{i=1}^{n}p_{\e_i} \geq kp_{\beta}$ for some $k > 0$.  Therefore, $p_{\beta}$ is in the ideal generated by $p_{\e_1}Bp_{\e_1p_{\e_1}}$ in $B$.  Using the inductive argument in Lemma \ref{lem:Iminimal}, we see that if $\gamma \in V\setminus V_{\geq}$ and $\gamma \neq \alpha$, then $\gamma$ is also in the ideal in $B$ $p_{\e_1}Bp_{\e_1}$.  It follows from this that $K_{0}(I)$ is generated by $K_{0}(p_{\e_1}C^{*}(X_{\e_{1}}^{*}X_{\e_{1}}p_{\e_1}))$ and $K_{0}(p_{\e_1}Bp_{\e_1})$.

If $\beta \in V_{\geq}$ then $\tr(p_{\e_i}) < \mu(\beta)$ so it follows that $p_{\e_i}$ is equivalent to $p_{s(\e_{i})}$ hence $p_{s(\e_{i})}$ is in the ideal in $B$ generated by $p_{\e_1}Bp_{\e_1}$.  The inductive argument form Lemma \ref{lem:Iminimal} will conclude that $K_{0}(I)$ is generated by $K_{0}(p_{\e_1}C^{*}(X_{\e_{1}}^{*}X_{\e_{1}}p_{\e_1})$ and $K_{0}(p_{\e_1}Bp_{\e_1}))$.

\underline{Case 3}:  By considering $\e_{2}, \dots, \e_{n}$ as in the previous case, the exact same proof shows that $K_{0}(I)$ is generated by $K_{0}((p_{\alpha} +q_{\alpha})C^{*}(X_{\e_{1}}^{*}X_{\e_{1}}(p_{\alpha} +q_{\alpha}))$ and $K_{0}((p_{\alpha} + q_{\alpha})\cB'(p_{\alpha} + q_{\alpha}))$
\end{proof}

Since $I$ has stable rank one, we obtain the following corollary.
\begin{cor}
Suppose $p$ and $q$ are projections in $M_{n}(I)$ satisfying $\Tr_{n}(p) < \Tr_{n}(q)$ where, $\Tr_{n}$ is the canonical trace on $M_{n}(I)$ induced from the trace on $I$.  Then there is a $v \in M_{n}(I)$ satisfying
$$
v^{*}v = p \text{ and } vv^{*} \leq q
$$
\end{cor}

\subsection{Extension to infinite graphs}

Assume that $\Gamma$ is a countably infinite connected graph with a countable edge set.  We may write $\Gamma$ as in increasing union of finite subgraphs $\Gamma_{n}$.   Since simplicity, $K$-groups, stable rank 1, and unique trace are preserved under inductive limits, we have the following corollary form our work in the previous section.

\begin{cor}\label{cor:infinite} Let $I$ be the ideal in $\cS(\Gamma, \mu)$ generated by $\set{X_\e}{\e \in E(\vec{\Gamma})}$.  Then the following statements hold
\be

\item $I$ is simple, contains the set $\set{p_{\beta}}{\beta \in V \setminus V_{\geq}}$, and does not intersect $\set{p_{\gamma}}{\gamma \in V_{\geq}}$.  Since $\Gamma$ infinite and connected implies that $V \setminus V_{\geq}$ is infinite, $I$ is not unital.

\item $\cS(\Gamma, \mu)/I \cong \bigoplus_{\gamma \in V_{\geq}} \C$

\item $K_{0}(I) \cong \Z\set{[p_{\beta}]}{\beta \in V \setminus V_{\geq}}$, $K_{1}(I)  = \{0\}$, and $K_{0}(I)^{+} = \set{x \in K_{0}(I)}{\Tr(x) > 0} \cup \{0\}$. 

\item $I$ and $\cS(\Gamma, \mu)$ have stable rank 1.

\item $I$ has a unique (up to scaling) lower semicontinuous tracial weight.  This weight is finite if and only if the support projection of $I$ in $\cM(\Gamma, \mu)$, the von Neumann algebra generated by $\cS(\Gamma, \mu)$, has finite trace.

\ee
\end{cor}


\section{Free differentials, atomless loops, and algebracity}\label{sec:poly}

For this section, we assume that $\Gamma$ is finite with a fixed weighting, $\mu$ on $V(\Gamma)$.  We let $\cA = \C\langle (X_{\e})_{\e \in E(\vec{\Gamma})}, (p_{\beta})_{\beta \in V(\Gamma)}\rangle$.  We aim to prove the following theorems, which are in the spirit of \cite{MR3356937,1407.5715,1408.0580}.

\begin{thm}\label{thm:nonzero}
Suppose $Q \in \cA$ is such that $\mu(\alpha) = \min\{\mu(\beta) : \beta \in V(\Gamma)\}$, and $Qp_{\alpha} = Q$.  Then if $a = a^{*} = p_{\alpha} a p_{\alpha} \in W^{*}(\cS(\Gamma, \mu), \Tr)$ and $Qa = 0$, then either $a = 0$ or $Q = 0$.
\end{thm}

\begin{thm}\label{thm:algebraic}
Let $x \in M_{n}(\cA)$ be self-adjoint.  Then the law of $x$ with respect to $\Tr \otimes \tr_{n}$ has algebraic Cauchy transform.
\end{thm}

Theorem \ref{thm:nonzero} has the following corollary.
\begin{cor}\label{cor:atomless}
Suppose $Q = Q^{*} \in \cA$, $p_{\alpha}$ is as in the statement of Theorem \ref{thm:nonzero}, and $Q = Qp_{\alpha} (= p_{\alpha}Qp_{\alpha})$.  Then the law of $Q$ in $(p_{\alpha}\cS(\Gamma, \mu)p_{\alpha}, \Tr)$ has no atoms.
\end{cor}

Notice that any such $P$ described in Corollary \ref{cor:atomless} must be a linear combination of elements of the form $X_{\e_{1}}\cdots X_{\e_{n}}$ where $\e_{1}\cdots\e_{n}$ is a loop in $\vec{\Gamma}$ based at $\alpha$.

To develop the machinery to prove Theorem \ref{thm:nonzero}, we will develop a free differential calculus along the lines of \cite{MR1228526,1411.0268}.

\subsection{Commutation with finite rank operators}

Let $\cH = L^{2}(\cS(\Gamma, \mu), \Tr)$ and observe that $\cH \cong\cX(\Gamma) \otimes_{\cC} \ell^{2}(V(\Gamma), \mu)$.  We will assume $\cH$ has inner product $\langle \cdot | \cdot \rangle_{\cH}$ which is linear in the \emph{right} variable.  Let $J: \cH \rightarrow \cH$ be the modular conjugation, (i.e. the isometric extension of $J(\widehat{x}) = \widehat{x^{*}}$ for $x \in \cS(\Gamma, \mu)$).  If we set $\cM = \cM(\Gamma, \mu)$, (the von-Neumann algebra generated by $\cS(\Gamma, \mu)$ acting on $\cH$), then $J\cM J = \cM'$, the commutant of $\cM$.

Note that $\cH$ is spanned by paths in $\vec{\Gamma}$ of finite length, with length 0 paths simply being vertices.  According to the euclidian structure above, we have 
$$
\| \e_{1}\cdots \e_{n} \|_{\cH} = \sqrt{\mu(t(\e_{n}))} \text{ and } J(\e_{1}\cdots \e_{n}) = \sqrt{\frac{\mu(t(\e_{n}))}{\mu(s(\e_{1}))}} \e_{n}^{\op}\cdots\e_{1}^{\op}.
$$
For $\xi, \eta \in \cH$, consider the rank-one operator $|\xi\rangle\langle \eta|$ which is defined by
$$
|\xi\rangle\langle \eta| (\zeta) = \xi \langle \eta|\zeta\rangle_{\cH}.
$$
We have the following lemma:
\begin{lem}\label{lem:com1}
$[\ell(\e), JX_{\e}J] = -\frac{1}{\sqrt[4]{\mu(s(e))^{3}\cdot \mu(t(e))}} |s(\e)\rangle\langle t(\e)|$, and $[\ell(\e), JX_{\e'}J] = 0$ if $\e \neq \e'$.

\end{lem}

\begin{proof}
It is straightforward to see that $\ell(\e) JX_{\e}J (\e_{1}\cdots \e_{n}) = JX_{\e}J\ell(\e) (\e_{1}\cdots\e_{n})$ if $n \geq 1$ and that $\ell(\e)JX_{\e}J (p_{\alpha}) = 0 = JX_{\e}J \ell(\e)(p_{\alpha})$ if $\alpha \neq t(\e)$.  Finally, it is also straightforward to check that
$$
[\ell(\e), JX_{\e}J](t(\e)) = -\left(\frac{\mu(t(\e))}{\mu(s(\e))}\right)^{3/4}s(\e).
$$
The result follows.
\end{proof}

We will now prove a commutation lemma which is central to our theorem.  The idea of this proof was communicated to the author by Dima Shlyakhtenko, who has a similar proof for when each $X_{\e}$ is a (scalar-valued) semicircular element.

\begin{lem}\label{lem:com2}
For each $\e \in \vec{E}$, let $R_{\e} \in \cB(\cH)$ be a finite-rank operator with $R_{\e}: Jp_{s(\e)}J\cH \rightarrow Jp_{t(\e)}J\cH$ (i.e. $\ker(R_{\e}) \supset (Jp_{s(\e)}J\cH)^{\perp}$ and $\ran(R_{\e}) \subset Jp_{t(\e)}J\cH$) .  Suppose further that 
$$
\sum_{\e \in \vec{E}} [R_{\e}, JX_{\e}J] = 0.
$$
Then $R_{\e} = 0$ for all $\e \in \vec{E}$.
\end{lem}

\begin{proof}
Let $a, b \in \cM$.  Since $J\cM J = \cM'$, we have
$$
0 = \sum_{\e \in \vec{E}} a[R_{\e}, JX_{\e}J]b =  \sum_{\e \in \vec{E}} [aR_{\e}b, JX_{\e}J]
$$
Fixing some $\e' \in \vec{E}$, and letting $\Tr_{\cH}$ be the trace on the finite rank operators, we have, with the aid of Lemma \ref{lem:com1}:
\begin{align*}
0 &= \Tr_{\cH}\left(\sum_{\e \in \vec{E}} \ell(\e') [aR_{\e}b, JX_{\e}J]\right)\\
&= \sum_{\e \in \vec{E}} \Tr_{\cH}(\ell(\e')aR_{\e}bJX_{\e}J - \ell(\e')JX_{\e}JaR_{\e}b)\\
&= \sum_{\e \in \vec{E}} \Tr_{\cH}(aR_{\e}b[JX_{\e}J\ell(\e') - \ell(\e')JX_{\e}J])\\
&= \frac{1}{\sqrt[4]{\mu(s(\e'))^{3}\cdot \mu(t(\e'))}}  \Tr_{\cH}(aR_{\e'}b | s(\e') \rangle \langle t(\e') | )\\
&= \frac{1}{\sqrt[4]{\mu(s(\e'))^{3}\cdot \mu(t(\e'))^{3}}} \langle t(\e') |  aR_{\e'}b s(\e')\rangle_{\cH}\\
&= \frac{1}{\sqrt[4]{\mu(s(\e'))^{3}\cdot \mu(t(\e'))^{3}}} \langle a^{*}t(\e') |  R_{\e'}b s(\e')\rangle_{\cH}.
\end{align*}
As this holds for all $a, b \in \cM$ this implies $R_{\e'} = 0$ since $R_{\e'}: Jp_{s(\e')}J\cH \rightarrow Jp_{t(\e')}J\cH$.

\end{proof}


\subsection{A free differential calculus for $\cA$}
For each edge $\e \in \vec{E}$, we define $\partial_{\e}: \cA \rightarrow \cA \otimes \cA$ to be the derivation which is the unique extension of 
$$
\partial_{\e}(X_{\e'}) = (\delta_{\e, \e'})p_{s(\e)} \otimes p_{t(\e)}. 
$$
Also see \cite{1411.0268}. Notice that this means that if $\e_{1}\cdots \e_{n}$ is a path in $\vec{\Gamma}$, then 
$$
\p_{\e}(X_{\e_{1}}\cdots X_{\e_{n}}) = \sum_{\e_{j} = \e}X_{\e_{1}}\cdots X_{\e_{j-1}} \otimes X_{\e_{j+1}}\cdots X_{\e_{n}},
$$
with the understanding that if $\e_{1} = \e$, then the term $p_{s(\e_{1})} \otimes X_{\e_{2}}\cdots X_{\e_{n}}$ appears and if $\e_{n} = \e$, then the term $X_{\e_{1}}\cdots X_{\e_{n-1}} \otimes p_{t(\e_{n})}$ appears.  Although the following lemma will not be needed in what follows, it demonstrates that the differential operators $\p_{\e}$ have a finite free Fisher information type property \cite{MR1228526}.

\begin{lem}\label{lem:closable}
\begin{enumerate}

\item For any $P \in \cA$, we have $(\Tr \otimes \Tr)(\p_{\e} P) = \sqrt{\mu(s(\e))\mu(t(\e))}\Tr(X_{\e^{\op}}P)$, i.e. $\sqrt{\mu(s(\e))\mu(t(\e))} X_{\e^{\op}}$ is a conjugate variable for $\p_{e}$.

\item If we define $\sigma \in \Aut(\cA \otimes \cA)$ by linear extension of $\sigma(a \otimes b) = b \otimes a$, then for any $P \in \cA$,
$$
(\p_{\e}(P))^{*} = \sigma(\p_{\e^{op}}(P^{*}))
$$

\item $\p_{\e}$ has a densely defined adjoint, $\p_{\e}^{*}$, which is given on individual tensors by
\begin{align*}
\p_{\e}^{*}(Q \otimes R)) &=  \sqrt{\mu(s(\e))\mu(t(\e))}QX_{\e}R - (\idd \otimes \Tr)(\p_{\e^{\op}}(Q))R - Q(\Tr \otimes \idd)(\p_{\e^{\op}}(R)). 
\end{align*}
Therefore, the operators $\p_{\e}$ are closable as operators from $\cH \rightarrow \cH \otimes \cH$. 
\end{enumerate}

\end{lem}

\begin{proof}
\be 

\item If $P$ is a monomial, then this immediately follows from the definition of $\p_{\e}$ and Lemma \ref{lem:rec}.  The linearity of $\p_{\e}$ makes the statement hold for all $P$.

\item If $P$ is a monomial, this immediately follows from the definition of $\sigma$ and $\p_{e}$.  The rest follows by linearity.

\item  Set $\tilde{X_{\e}} = \sqrt{\mu(s(\e))\mu(t(\e))}X_{\e}$ and $\tilde{X}_{\e^{op}} = \sqrt{\mu(s(\e))\mu(t(\e))}X_{\e^{\op}}$. For $P, Q, R \in \cA$, we have:
\begin{align*}
\langle \widehat{Q \otimes R} | \widehat{\p_{e}(P)} \rangle_{\cH \otimes \cH} &= \Tr \otimes \Tr ( (Q^{*} \otimes R^{*})\p_{\e}(P))\\
&= \Tr \otimes \Tr (Q^{*}\p_{\e}(P)R^{*})\\
&= \Tr \otimes \Tr (\p_{e}(Q^{*}PR^{*}) - \p_{e}(Q^{*})PR^{*} - Q^{*}P\p_{e}(R^{*}))\\
&= \Tr(\tilde{X}_{\e^{\op}}Q^{*}PR^{*}) - \Tr((\Tr \otimes \idd)(\p_{\e}(Q^{*}))PR^{*}) - \Tr(Q^{*}P(\idd \otimes \Tr)(\p_{\e}(R^{*})))\\
&= \Tr(R^{*}\tilde{X}_{\e^{\op}}Q^{*}P) - \Tr(R^{*}(\Tr \otimes \idd)(\p_{\e}(Q^{*}))P) - \Tr((\idd \otimes \Tr)(\p_{\e}(R^{*}))Q^{*}P)\\
&= \Tr((Q\tilde{X_{\e}}R)^{*}P) - \Tr([(\Tr \otimes \idd)(\sigma(\p_{\e^{\op}}(Q))^{*})^{*} R]^{*} P)\\
&- \Tr([Q(\idd \otimes \Tr)(\sigma(\p_{e^{\op}}(R))^{*})^{*}]^{*}P)\\
&= \Tr((Q\tilde{X_{\e}}R)^{*}P) - \Tr([(\id \otimes \Tr)(\p_{\e^{\op}}(Q))R]^{*}P) - \Tr([Q(\Tr \otimes \idd)(\p_{\e^{\op}}(R)]^{*}P)\\
&= \langle \widehat{Q\tilde{X_{\e}}R} - (\id \otimes \Tr)(\p_{\e^{\op}}(Q))\widehat{R} - Q(\Tr \otimes \idd)\widehat{(\p_{\e^{\op}}(R))} | \widehat{P}\rangle_{\cH}
\end{align*}

\ee

\end{proof}

Note that $\cM \otimes \cM$ has a canonical $\cM-\cM$ bimodule structure given by $a(x_{1} \otimes x_{2})b = ax_{1} \otimes x_{2}b$.  Moreover, we can realize $\cM \otimes_{alg} \cM$ as a subalgebra of the finite rank operators on $\cH$ by $x \otimes y \mapsto | \widehat{x} \rangle \langle \widehat{y^*}|$. These two actions are compatible, i.e. under the identification of $\cM \otimes_{alg} \cM$ with finite rank operators on $\cH$, we have
$$
a(x \otimes y)b \mapsto |a\widehat{x}\rangle \langle b^{*}\widehat{y^*}| = a | \widehat{x} \rangle \langle \widehat{y^*} | b
$$
where the last product is the product in $\cB(\cH)$.

We will occasionally realize the $\cM-\cM$ bimodule action on $\cM \otimes \cM$ as an $\cM \otimes \cM$ action on $\cM \otimes \cM$ via the linear extension of $(a \otimes b) \# (x \otimes y) = ax \otimes yb$.  We now prove the key inductive lemma which is along the lines of \cite{1407.5715}.

\begin{thm}\label{thm:zero}
Suppose $Q \in \cA$ and there are $\alpha, \beta \in V$ with $p_{\alpha}Q = Q = Qp_{\beta}$.  Suppose further that there are $a$ and $b$ in $\cM$ with $aQ = 0 = Qb$.  Then for each $\e \in \vec{E}$, $a\p_{\e}(Q)b = 0$.
\end{thm}

\begin{proof}

The assumptions in the problem imply that 
$$
a(Q \otimes p_{\beta} - p_{\alpha} \otimes Q)b = 0.
$$
The identity
$$
Q \otimes p_{\beta} - p_{\alpha} \otimes Q = \sum_{\e \in \vec{E}} \p_{\e}(Q) \# (X_{\e} \otimes 1 - 1 \otimes X_{\e})
$$
can be observed by noting that if $Q$ is a monomial $X_{\e_{1}}\cdots X_{\e_{n}}$ with $s(\e_{1}) = \alpha$ and $t(\e_{n}) = \beta$, then
\begin{align*}
\sum_{\e \in \vec{E}}  \p_{\e}(Q) \# (X_{\e} \otimes 1) &= Q \otimes p_{\beta} + \sum_{k=1}^{n} X_{\e_{1}}\cdots X_{\e_{k}} \otimes  X_{\e_{k+1}}\cdots X_{\e_{n}} \text{ and }\\
\sum_{\e \in \vec{E}}  \p_{\e}(Q) \# (1 \otimes X_{\e}) &= p_{\alpha} \otimes Q + \sum_{k=1}^{n} X_{\e_{1}}\cdots X_{\e_{k}} \otimes  X_{\e_{k+1}}\cdots X_{\e_{n}}
\end{align*}
We therefore have
\begin{align*}
\sum_{e \in \vec{E}} (a \otimes b) \# \p_{\e}(Q) \# (X_{\e} \otimes 1 - 1 \otimes X_{\e}) &= 0 \text{ which implies }\\
\sum_{e \in \vec{E}} [a\p_{\e}(Q)b, JX_{\e^{\op}}J] &= 0
\end{align*}
where the last equation is with $a\p_{\e}(Q)b$ viewed as a finite rank operator on $\cH$.  Note that under this identification, $a\p_{\e}(Q)b : Jp_{s(\e^{\op})}J\cH \rightarrow Jp_{t(\e^{\op})}J\cH$, so applying Lemma \ref{lem:com2} (for $\e^{\op}$) to $a\p_{\e}(Q)b$  gives $a\p_{\e}(Q)b = 0$ as desired.
\end{proof}

We now prove Theorem \ref{thm:nonzero}.  The proof will be an amalgamated version of the proof of the corresponding theorem in \cite{1407.5715}.

\begin{proof}[Proof of Theorem \ref{thm:nonzero}]
Suppose $p_{\alpha}$ is in the statement of Theorem \ref{thm:nonzero} with $Q \in \cA$ nonzero and $Qp_{\alpha} = Q$.  Suppose further that $Qa = 0$ for $a = a^{*}$ and $p_{\alpha}ap_{\alpha} = a$.  Define the degree of $Q$ to be the length of the longest monomial in the $X_{\e}$'s with nonzero coefficient in the expansion of $Q$ into monomials.  If $Q$ is a linear combination of the elements $\set{p_{\gamma}}{\gamma \in V}$, then $Q$ is said to have degree 0.  We will prove the result by induction on the degree of $Q$.  The result is trivial for $Q$ of degree 0 so we will now assume that the degree of $P$ is at least 1.

By expanding $Q = \sum_{\beta \in V} p_{\beta}Q$, we see that $Qa = 0$ if and only if $p_{\beta}Qa = 0$ for all $\beta$, so we assume that $Q = p_{\beta}Q$ for some $\beta \in V$.  Let $\tilde{q}$ be the projection onto the kernel of $Q$.  Since $Qp_{\alpha} = Q$, $\Tr(\tilde{q}) \geq \sum_{\gamma \in V\setminus\{\alpha\}} \mu(\gamma)$.  Assume that $a \neq 0$.  Then since $a$ is supported under $p_{\alpha}$ we have $\Tr(\tilde{q}) > \sum_{\gamma \in V\setminus\{\alpha\}} \mu(\gamma)$.  Let $q$ be the projection onto the kernel of $P^{*}$.  Since $\Tr(q) = \Tr(\tilde{q})$, it follows from minimality of $\alpha$, that $qp_{\gamma} \neq 0$ for all $\gamma \in V$.  In particular, $qp_{\beta} \neq 0$.

Note that $qP = 0$.  By Theorem \ref{thm:zero}  $q\p_{\e}(Q)a = 0$.  This implies that 
$$
0 = (\Tr \otimes \idd)(q\p_{\e}(Q)a) = (\Tr \otimes \idd)(q\p_{\e}(Q))a = 0
$$
for all $\e$.  Choose $\e$ so that $X_{\e}$ is the leftmost term in at least one monomial in the sum representing $P$. We have $\Tr(qp_{\beta}) \neq 0$, so it follows that $(\Tr \otimes \idd)(q'\p_{\e}(Q)) \in \cA$, is nonzero, has right support under $p_{\alpha}$, and has degree strictly less than that of $Q$.  Since $a \neq 0$, this contradicts the assumption that for any nonzero $P$ of strictly smaller degree than $Q$ with right support under $p_{\alpha}$, $Pa = 0$ implies $a = 0$.  Therefore, $a$ had to already be zero in the previous paragraph.
\end{proof}

\subsection{Algebracity}\label{sec:alg}

We will now turn toward proving Theorem \ref{thm:algebraic}.  The majority of the discussion below comes from and is inspired by \cite{MR3356937}.  To begin, we need to set up some notation.

\begin{defn}
\begin{enumerate}
\item If $R$ is a ring, we define the ring of formal power series in the variables $z_{1}, \dots, z_{n}$ (denoted as $R[[z_{1}, \dots, z_{n}]]$) to be the set of formal sums of the form 
$$
P =  \sum_{k_{1}, \dots k_{n} \geq 0}  P_{k_{1}, \dots, k_{n}}z_{1}^{k_{1}}\cdots z_{n}^{k_{n}}
$$
where $P_{k_{1}, \dots, k_{n}} \in \R$.  Addition of two power series is defined term-wise.  Multiplication is defined by
$$
(PQ)_{k_{1}, \dots, k_{n}} = \sum_{j = 1}^{n} \sum_{\ell_{j} = 0}^{k_{j}} P_{\ell_{1}, \dots, \ell_{n}}Q_{k_{1} - \ell_{1}, \dots, k_{n} - \ell_{n}}
$$ 
We will let $R[z_{1}, \dots, z_{n}]$ denote the polynomials in the variables $z_{1}, \dots, z_{n}$.  Note that $P \in R[z_{1}, \cdots, z_{n}]$ if and only if $P_{k_{1}, \dots, k_{n}} = 0$ for all but finitely many $(k_{1}, \dots, k_{n})$.

\item If $R$ is an integral domain, we say that $P \in R[[z_{1}, \dots, z_{n}]]$ is \emph{algebraic} if there exist $Q_{0}, Q_{1}, \dots, Q_{n} \in R[z_{1}, \dots, z_{n}]$ not all zero satisfying 
$$
\sum_{j=0}^{n} Q_{j}P^{j} = 0
$$
The algebraic elements in $R[[z_{1}, \dots, z_{n}]]$ form a ring.  This ring will be denoted as $R_{alg}[[z_{1}, \dots, z_{n}]]$.
\end{enumerate}  
\end{defn}

Given Neumann algebra, $\cN$, a faithful positive linear functional $\phi$ on $\cN$, and a self-adjoint $a \in \cN$, there exists a unique positive measure $\mu_{a}$, supported on the spectrum of $a$, which satisfies
$$
\phi(a^{n}) = \int_{\R} x^{n}d\mu_{a}(x).
$$
Recall that the \emph{Cauchy Transform} of a measure, $\mu$, is defined as 
$$
G_{\mu}(z) = \int_{\R}\frac{d\mu(t)}{z - t}.
$$
If $\mu$ is compactly supported, it is straightforward to see that $G_{\mu}(z)$ has a Laurent expansion about $z = 0$, and that $\lim_{z \rightarrow \infty} G_{\mu}(z) = 0$.  Therefore the Laurent series for $G_{\mu}(z)$ is an element of $\C[[\frac{1}{z}]]$ so it makes sense to ask if $G_{\mu_{a}}$ is algebraic for any $a \in \cA$.

In order to answer this question, we will need a notion of rational power series
\begin{defn}
\begin{enumerate}
\item Let $R$ be a unital subring of the unital ring, $S$.  We say that $R$ is rationally closed if whenever $A$ is an $n\times n$ matrix with entries in $R$ and invertible in $M_{n}(S)$, then $A^{-1} \in M_{n}(R)$.  The rational closure of $R$ is the smallest rationally closed subring of $S$ containing $R$.

\item A power series $P \in R[[z_{1}, \dots, z_{n}]]$ is rational if $P$ is in the rational closure of $R[z_{1}, \dots, z_{n}]$, viewed as a subring of $ R[[z_{1}, \dots, z_{n}]]$.  The set of rational power series will be denoted as $R_{rat}[[z_{1}, \dots, z_{n}]]$
\end{enumerate}
\end{defn}

The following lemma from \cite{MR3356937} will be helpful.

\begin{lem}[\cite{MR3356937}]\label{lem:SS1}
Suppose that  $\cN$ is a von Neumann algebra with faithful trace, $\tau$.  Given $x_{1}, x_{2}, \dots, x_{n}, \dots \in \cN$, and $\sum_{n = 0}^{\infty}x_{n}z^{n}$, we define 
$$
\Tr_{\cN}\left(\sum_{n = 0}^{\infty}x_{n}z^{n}\right) = \sum_{n=0}^{\infty} \tau(x_{n})z^{n} \subset \C[[z]].
$$
Suppose that $\cA$ is a subalgebra of $\cN$ and that
$$
\Tr_{\cN}(\cA_{rat}[[z]]) \subset \C_{alg}[[z]],
$$
then for every self-adjoint matrix $A \in M_{n}(\cA)$, $G_{\mu_{A}}$ is algebraic.
\end{lem}

As above we assume that $\Gamma$ is a finite graph, and we fix a weighting $\mu$ on $V(\Gamma)$.  Let $\cM = \cM(\Gamma, \mu)$ and $\cA = \C\langle \cC, (X_{\e})_{\e \in \vec{E}}\rangle$.  In order to verify the hypotheses in Lemma \ref{lem:SS1} we need to make use of power series in non-commuting variables.

\begin{defn}
\begin{enumerate}
\item Let $R$ be a ring, and $X = \{x_{1}, \dots, x_{n}\}$ be a finite set, often called an \emph{alphabet}.  A \emph{word} in $X$ is a finite string $x_{i_{1}}, \dots x_{i_{k}}$ and the set of all words in $X$ will be denoted as $W(X)$, and the empty word will be denoted $\mathbf{1}$.

\item The non-commutative power series ring, $R\langle \langle X \rangle \rangle$ consists of formal sums of the form
$$ 
P = \sum_{w \in W(X)}P_{w}w 
$$
for $P_{w} \in R$.  Addition of elements in $R\langle \langle X \rangle \rangle$ is done coordinate wise.  Multiplication of elements $P$ and $Q$ in $R\langle \langle X \rangle \rangle$ is defined as follows
$$
(PQ)_{w} = \sum_{\substack{ u, v \in W(X) \\ w = uv}} P_{u}Q_{v}
$$
The non-commutative polynomials in $X$, denoted by $R\langle X \rangle$ is the subring of $R\langle \langle X \rangle \rangle$ consisting of elements of the form
$$
P = \sum_{w \in W(X)}P_{w}w 
$$
where $P_{w} = 0$ for all but finitely many $w \in W(X)$

\item If $Z = \{z_{1}, \dots, z_{m}\}$ is an alphabet disjoint from $X$, a proper algebraic system over $R$ is a set of equations
$$
z_{i} = p_{i}(x_{1}, \dots, x_{n}, z_{1}, \dots, z_{m})
$$
where $p_{i} \in R\langle X \cup Z \rangle$ has no constant term, nor any term of the form $\alpha z_{j}$ for $\alpha \in R$ and $j \in \{1, \dots, m\}$.  A \emph{solution} to a proper algebraic system is $(P_{1}, \dots, P_{m}) \in R\langle \langle X \rangle \rangle^{m}$ with $(P_{i})_{\mathbf{1}} = 0$, satisfying
$$
P_{i} = p_{i}(x_{1}, \dots, x_{n}, P_{1}, \dots, P_{m})
$$
for each each $i \in \{1, \dots, m\}$.

\item We say that $P \in  R\langle \langle X \rangle \rangle$ is \emph{algebraic} if $P - P_{\mathbf{1}}\mathbf{1}$ is a component of a solution of a proper algebraic system.  The set of the algebraic elements in $R\langle \langle X \rangle \rangle$ will be denoted by  $R_{alg}\langle \langle X \rangle \rangle$.  $R_{alg}\langle \langle X \rangle \rangle$ is a subring of $R\langle \langle X \rangle \rangle$. 

\item We say that $P \in  R\langle \langle X \rangle \rangle$ is \emph{rational} if $P$ is in the rational closure of $R\langle X \rangle$ in $R\langle \langle X \rangle \rangle$.  The set of rational elements in $R\langle \langle X \rangle \rangle$ will be denoted as $R_{alg}\langle \langle X \rangle \rangle$
\end{enumerate}
\end{defn}

Fix an alphabet $\mathscr{X} = \set{\e}{\e \in \vec{E}} \cup \set{\alpha}{\alpha \in V}$, and let $\cX = \set{X_{\e}}{\e \in \vec{E}} \cup \set{p_{\alpha}}{\alpha \in V}$ so that $\cA = \C\langle \cX \rangle$.  If $w = \e_{1}\dots \e_{n}$, then set $w(\cX) = X_{\e_{1}}\cdots X_{\e_{n}}$.   For $\alpha \in V$, set $P^{\alpha} \in \C\langle \langle \mathscr{X} \rangle \rangle$ to be the following element:
$$
P^{\alpha} = \sum_{w \in W(\mathscr{X})} P^{\alpha}_{w}w \text{ where } 
\begin{cases} P^{\alpha}_{\mathbf{1}} &= 0\\
P^{\alpha}_{\gamma} &= \delta_{\gamma, \alpha}\mu(\alpha) \text{ if } \gamma \in V\\
P^{\alpha}_{w} &= 0 \text{ if } p_{\alpha}w(\cX)p_{\alpha} = 0 \\
P^{\alpha}_{w} &= 0 \text{ if } |w| \geq 2 \text{ and } w \text{ contains } \gamma \in V\\
P^{\alpha}_{w} &= \Tr(w(\cX)) \text{ if } p_{\alpha}w(\cX)p_{\alpha} = w(\cX) \neq 0 \text{ and } w = \e_{1}\dots\e_{n}
\end{cases}
$$
Let $L(\alpha)$ denote the loops in $\vec{\Gamma}$ which are based at $\alpha$.   A key point in our analysis is the following lemma.  
 
 \begin{lem}
 $P^{\alpha} - \mu(\alpha)\alpha$ is algebraic for each $\alpha \in V$.
 \end{lem}

\begin{proof}

We note that 
$$
P^{\alpha} - \mu(\alpha)\alpha = \sum_{\e_{1}\dots\e_{n} \in L(\alpha)} \Tr(X_{\e_{1}}\dots X_{\e_{n}})\e_{1}\cdots \e_{n}
$$
Using Lemma \ref{lem:rec}, this series can be rewritten as:
\begin{align*}
& \sum_{\substack{t_{\e_{n}} = \alpha \\ \e_{1}\dots\e_{k-1} \in L(\alpha) \\ \e_{k+1}\dots \e_{n-1} \in L(s(\e_{n-1}))}}  \frac{1}{\sqrt{\mu(s(\e_{n})) \mu(\alpha)}}[\Tr(X_{\e_{1}}\cdots X_{\e_{k-1}})\e_{1}\cdots \e_{k-1}]\e_{n}^{\op}[\Tr(X_{\e_{k+1}}\cdots X_{\e_{n-1}}) \e_{k+1}\cdots \e_{n-1}]\e_{n}\\
&+  \sum_{\e_{1}\dots\e_{n-2} \in L(\alpha)}  \sqrt{\frac{\mu(\alpha)}{\mu(s(\e_{n}))}} [\Tr(X_{\e_{1}}\cdots X_{\e_{n-2}})\e_{1}\dots \e_{n-2}]\e_{n}^{\op}\e_{n}\\
& +  \sum_{\e_{2}\dots\e_{n-1} \in L(s(\e_{n}))}  \sqrt{\frac{\mu(s(\e_{n}))}{\mu(\alpha)}} \e_{n}^{\op}[\Tr(X_{\e_{2}}\cdots X_{\e_{n-1}})\e_{2}\dots \e_{n-1}]\e_{n}\\
& + \sqrt{\mu(\alpha)\mu(s(\e))}\e_{n}^{\op}\e_{n}
\end{align*}
This means that $(P^{\alpha} - \mu(\alpha)\alpha)_{\alpha \in V}$ is a solution to the proper algebraic system
$$
z_{\alpha} = \sum_{\beta \sim \alpha} \sum_{\substack{t(\e) = \alpha \\ s(\e) = \beta}}  \sqrt{\mu(\alpha)\cdot\mu(\beta)}\e^{\op}\e + \sqrt{\frac{\mu(\alpha)}{\mu(\beta)}}  z_{\alpha}\e^{\op}\e +\frac{1}{\sqrt{\mu(\beta)\mu(\alpha)}}z_{\alpha}\e^{\op}z_{\beta}\e +   \sqrt{\frac{\mu(\beta)}{\mu(\alpha)}} \e^{\op} z_{\beta}\e
$$
Where $\alpha$ ranges through all of $V$.
\end{proof}

\begin{cor}
$P^{\alpha}$ is algebraic for all $\alpha \in V$. 
\end{cor}
We will now let $P \in \cA_{rat}[[z]]$, and we set $P^{\Gamma} = \sum_{\alpha \in V}P^{\alpha}$ and note that $P^{\Gamma} \in \C_{alg}\langle \langle X \rangle \rangle$.  There is a homomorphism $\pi: \C\langle \mathscr{X} \rangle \rightarrow \cA$ which is uniquely determined by $\pi(\e) = X_{\e}$ for all $\e \in \vec{E}$ and $\pi(\alpha) = p_{\alpha}$ for all $\alpha \in V$.  This map extends to a map (also denoted $\pi$) from $\C\langle \mathscr{X} \rangle [[z]] \rightarrow \cA [[z]]$.  We can find $\overline{P} \in \C\langle \mathscr{X} \rangle_{rat}[[z]]$ such that $\pi(\overline{P}) = P$.

In \cite{MR2048733} there is a canonical way to realize $\C\langle \mathscr{X} \rangle_{rat}[[z]] \subset \C(z)_{rat}\langle \langle \mathscr{X} \rangle \rangle$ where $\C(z)$ is the quotient field of $\C[z]$.  Note that we may realize $P^{\Gamma} \in \C(z)_{alg}\langle \langle \mathscr{X} \rangle \rangle$.  If we consider the Hadamard product
$$
\overline{P} \odot P^{\Gamma} = \sum_{w \in X} \overline{P}_{w}\cdot P^{\Gamma}_{w} \cdot w,
$$
then $\overline{P} \odot P^{\Gamma} \in \C(z)_{alg}\langle \langle X \rangle \rangle$ by \cite{MR0142781}.  

Observe that if we write $P = \sum_{m=0}^{\infty} [p_{m, \vec{E}}((X_{\e})_{\e \in \vec{E}}) + p_{m, V}((p_{\alpha})_{\alpha \in V})]z^{m}$ for polynomials $p_{m, \vec{E}}$ and $p_{m, V}$, then 
$$
\overline{P} = \sum_{m=0}^{\infty} [p_{m, \vec{E}}((\e)_{\e \in \vec{E}}) + p_{m, V}((\alpha)_{\alpha \in V}) + q_{m}((\e)_{\e \in \vec{E}}, (\alpha)_{\alpha \in V})]z^{m}
$$
where $q_{m}((X_{\e})_{\e \in \vec{E}}, (p_{\alpha})_{\alpha \in V}) = 0$.  If we let $\coeff(p, w)$ be the coefficient of $w$ in the expansion of $p$ into monomials, we see that 

\begin{align*}
\overline{P} \odot P^{\Gamma} &= \sum_{\alpha \in V} \mu(\alpha)\left(\sum_{m=0}^{\infty}[\coeff(p_{m, V}, \alpha) + \coeff(q_{m}, \alpha)]z^{m}\right)\alpha\\
&+ \sum_{\alpha \in V} \sum_{\e_{1}\dots \e_{n} \in L(\alpha)} \Tr(X_{\e_{1}}\cdots X_{\e_{n}})\left(\sum_{m=0}^{\infty}[\coeff(p_{m, \vec{E}}, \e_{1}\cdots \e_{n}) + \coeff(q_{m}, \e_{1}\cdots \e_{n})]z^{m}\right)\e_{1}\dots \e_{n} 
\end{align*}

If one inserts 1 for each $\e$ and $\alpha$ then we get
\begin{align*}
&\sum_{\alpha \in V} \mu(\alpha)\left(\sum_{m=0}^{\infty}[\coeff(p_{m, V}, \alpha) + \coeff(q_{m}, \alpha)]z^{m}\right)\\
&+ \sum_{\alpha \in V} \sum_{\e_{1}\dots \e_{n} \in L(\alpha)} \Tr(X_{\e_{1}}\cdots X_{\e_{n}})\left(\sum_{m=0}^{\infty}[\coeff(p_{m, \vec{E}}, \e_{1}\cdots \e_{n}) + \coeff(q_{m}, \e_{1}\cdots \e_{n})]z^{m}\right)\\
&= \sum_{m=0}^{\infty} \Tr(p_{m, \vec{E}}((X_{\e})_{\e \in \vec{E}}) + p_{m, V}((p_{\alpha})_{\alpha \in V}) + q_{m}((X_{\e})_{\e \in \vec{E}}, (p_{\alpha})_{\alpha \in V}))z^{m}\\
&= \sum_{m=0}^{\infty} \Tr(p_{m, \vec{E}}((X_{\e})_{\e \in \vec{E}}) + p_{m, V}((p_{\alpha})_{\alpha \in V})))z^{m}\\
&= \Tr_{\cM}(P)
\end{align*}
We see that $\Tr_{\cM}(P)$ is algebraic from \cite{MR2048733}.  This discussion proves Theorem \ref{thm:algebraic}.

\begin{cor}\label{cor:alg}  Let $Q \in \cA$ be self-adjoint with $p_{\alpha}Q p_{\alpha} = Q$ and $Q$ not a scalar multiple of $p_{\alpha}$.  If $\mu(\alpha) = \min\set{\mu(\beta)}{\beta \in V}$, then the law of $Q$ in $(p_{\alpha}\cS(\Gamma, \mu)p_{\alpha}, \Tr)$ is absolutely continuous with respect to Lebesgue measure.   Moreover, the spectrum of $Q$ in $p_{\alpha}\cS(\Gamma, \mu)p_{\alpha}$ is a finite union of disjoint closed intervals, each of which has measure in $(0, \mu(\alpha)] \cap \Z[\set{\mu(\beta)}{\beta \in V}]$. 

\end{cor}

Since the Cauchy transform $G_{\mu_{P}}$ is algebraic for any self-adjoint $P \in M_{n}(\cA)$, we can deduce results stating that the law of any positive $P \in M_{n}(\cA)$ can not have a significant portion of mass near 0 (provided $\mu_{P}$ has no atom at 0). The approach is exactly the same as \cite{MR3356937}, Theorem 5.17.  Given a finite measure, $\mu$ on $\R$, the spectral density function $F_{\mu}$ is defined by $F_{\mu}(t) = \mu((-\infty, t])$

\begin{thm}\label{thm:entropy}
Let $P \in M_{n}(\cA)$ be positive, and let $\alpha \in V$ satisfying $\mu(\alpha) = \min\set{\mu(\beta)}{\beta \in V}$.  Then the following hold
\begin{enumerate}  
\item $$\displaystyle \lim_{\delta \rightarrow 0^{+}} \int_{\delta}^{\|P\|} \frac{1}{t}(F_{\mu_{P}}(t) - F_{\mu_{P}}(0)) < \infty.$$

\item $$\displaystyle \lim_{\delta \rightarrow 0^{+}} \int_{\delta}^{\|P\|} \log(t) d\mu_{P}(t) > -\infty.$$  
In particular, if $P \in \cA$, and $p_{\alpha}Pp_{\alpha} = P$, and $\mu'_{P}$ is a the law of $P$ in $p_{\alpha}\cS(\Gamma, \mu)p_{\alpha}$, then
$$
 \int_{0}^{\|P\|} \log(t) d\mu'_{P}(t) > -\infty.
$$

\end{enumerate}
\end{thm}


\section{Applications}\label{sec:app}

\subsection{An application to Wishart matrices}

Recall that a Wishart matrix $A$ is an $M \times N$ random matrix with independent complex gaussian entries $a_{ij}$.  Moreover, the entries have the following covariances
\begin{align*}
E(a_{ij}\overline{a_{kl}}) &= \frac{1}{\sqrt{MN}}\delta_{ik}\delta_{jl}\\
E(a_{ij}a_{kl}) &= 0.
\end{align*}
\begin{nota}
For each integer, $n$, fix positive integers $M_{1}(n), \dots , M_{k}(n)$ with $M_{1}(n) = n$ for all $n$ and $M_{i}(n) \geq n$ for  each $i \in \{1, \dots , k\}$.  Assume further that $\lim_{n \rightarrow \infty} M_{i}(n)/M_{1}(n) = \gamma_{i}$.  Let $\set{A_{n_{ij}}}{1 \leq i, j \leq k}$ be a family of $M_{i}(n) \times M_{j}(n)$ Wishart matrices, and assume further that the entries of $A_{n_{ij}}$ are independent from the entries of $A_{n_{kl}}$ provided that $i \neq k$ or $j \neq l$. 
\end{nota} 
We are ready to state an application of our work to this family of Wishart matrices.

\begin{thm}\label{thm:approxlaw}
Let $Q$ be a noncommutative, nonconstant polynomial in the variables $(X_{ij})_{1 \leq i,j \leq n} \cup (X^{*}_{ij})_{1 \leq i,j \leq n}$.  Assume that $Q$ makes sense as a random matrix, $P_{n}$, when $A_{n_{ij}}$ and $A_{n_{ij}}^{*}$ are inserted for $X_{ij}$ and $X_{ij}^{*}$ respectively.  Assume further that $Q_{n} = Q_{n}^{*}$ and is of size $n \times n$.  Then
\begin{enumerate}
\item There is a unique compactly supported probability measure $\mu_{P}$ so that if $\tr_{n}$ is the (normalized) trace on the $n \times n$ complex matrices, then for each $m \in \N$,
$$
\tr_{n} \otimes E(Q_{n}^{m}) \rightarrow \int_{\R} x^{m}d\mu(x)
$$
as $n \rightarrow \infty$.

\item The measure $\mu_{P}$ has no atoms.  Moreover, the support of $\mu_{P}$ is a finite union of intervals where each interval has measure in the set $\Z[\set{\gamma_{i}}{1 \leq i \leq k}] \cap (0, 1]$.  In particular, the eigenvalues of $P$ do not cluster around any point.

\end{enumerate} 
\end{thm}

Let $(Y, \mu)$ be a probability space, set $M(n) = \sum_{m = 1}^{k} M_{m}(n)$, and let $\cB_{n} = M_{M(n)}(\C) \otimes L^{\infty}(Y, \mu)$ be the algebra of $M(n) \times M(n)$ complex random matrices. If $\Tr_{n}$ is the (non-normalized) trace on the algebra of $M(n) \times M(n)$ matrices, (satisfying $\Tr(I_{M(n)}) = M(n)$), then we define $\phi_{n}: \cB_{n} \rightarrow \C$ by
$$
\phi_{n}(A) = \frac{1}{n}(\Tr_{n} \otimes E)(A).
$$
We will view elements in $\cB_{n}$ as $k \times k$ block matrices where the $ij$ block is of size $M_{i}(n) \times M_{j}(n)$.  If we are given the Wishart matrices $A_{n_{ij}}$ as above, then we let $\tilde{A}_{n_{ij}} \in \cB_{n}$ be the matrix whose $ij$ block is $A_{n_{ij}}$ and whose other blocks are all zero.  Notice that if $P$ is as in the statement of Theorem \ref{thm:approxlaw} so that $P((A_{n_{ij}})_{1 \leq i, j, \leq k}, (A^{*}_{n_{ij}})_{1 \leq i, j, \leq k})$ is an $n \times n$ random matrix, then $P((\tilde{A}_{n_{ij}})_{1 \leq i, j, \leq k}, (\tilde{A}^{*}_{n_{ij}})_{1 \leq i, j, \leq k}) \in \cB_{n}$ has its $1,1$ block equal to $P((A_{n_{ij}})_{1 \leq i, j, \leq k}, (A^{*}_{n_{ij}})_{1 \leq i, j, \leq k})$ and all other blocks zero. 

To be able to prove Theorem \ref{thm:approxlaw}, we let $(\Gamma, V, E, \mu)$ be the complete graph on $k$ vertices with two unoriented edges, $e_{ij}$ and $e_{ji}$ having the distinct vertices $i$ and $j$ as endpoints.  Let $\mu: V \rightarrow \R_{+}$ given by $\mu(n) = \gamma_{n}$, and let $\vec{\Gamma}$ be the directed version of $\Gamma$.  The following theorem was proved in \cite{MR2732052}.

\begin{thm}\label{thm:GJSapprox}(\cite{MR2732052})
The elements $\tilde{A}_{n_{ij}}$ converge in $*$-distribution, under $\phi_{n}$, to the elements $X_{\e_{ij}}$.  More specifically, if $Q$ is any non-commutative polynomial in the variables $(X_{ij})_{1 \leq i, j \leq k}$ and $(X^{*}_{ij})_{1 \leq i, j \leq k}$, then
$$
\lim_{n \rightarrow \infty} \phi_{n}(Q((\tilde{A}_{n_{ij}})_{1 \leq i, j, \leq k}, (\tilde{A}^{*}_{n_{ij}})_{1 \leq i, j, \leq k})) = \Tr_{\cS(\Gamma, \mu)}(Q((X_{\e_{ij}})_{1 \leq i, j, \leq k}, (X_{\e^{\op}_{ij}})_{1 \leq i, j, \leq k})) 
$$

\end{thm}
Using this theorem, we can prove Theorem \ref{thm:approxlaw}

\begin{proof}[Proof of Theorem \ref{thm:approxlaw}]
In the algebra $\cB_{n}$, let $P_{1, n}$ be the matrix whose $1, 1$ block is the identity and all other blocks are zero and $Q$ is in the statement of Theorem \ref{thm:approxlaw}.  In $\cB_{n}$, If $\tilde{Q}_{n} = Q((\tilde{A}_{n_{ij}})_{1 \leq i, j, \leq k}, (\tilde{A}^{*}_{n_{ij}})_{1 \leq i, j, \leq k}$, then $P_{1, n}\tilde{Q}_{n}P_{1, n} = \tilde{Q}_{n}$, and the $1, 1$ block of this random matrix is $Q_{n} = Q(({A}_{n_{ij}})_{1 \leq i, j, \leq k}, ({A}^{*}_{n_{ij}})_{1 \leq i, j, \leq k})$.  By Theorem \ref{thm:GJSapprox}, 
$$
\tr(Q_{n}^{m}) = \int_{\R} x^{m}d\mu_{Q}(x)
$$
with $\mu_{Q}$ compactly supported.  Since the law of $Q((X_{\e_{ij}})_{1 \leq i, j, \leq k}, (X_{\e^{\op}_{ij}})_{1 \leq i, j, \leq k})) $ has no atoms in $p_{1}\cS(\Gamma, \mu)p_{1}$, it follows that $\mu_{Q}$ has no atoms.  Moreover the absolute continuity of $\mu_{Q}$ and the rest of statement 2 in Theorem \ref{thm:approxlaw} follows from Corollary \ref{cor:alg}.
\end{proof}

\subsection{An application to planar algebras}

Let $\cP_{\bullet}$ be a (sub)factor planar algebra.   For the definition and basic properties of (sub)factor planar algebras, see \cite{0902.1294, 1208.5505}.  We briefly recall the construction of \cite{MR2732052}.  Let $\Gr_{0}(\cP_{\bullet}) = \bigoplus_{n \geq 0} \cP_{n}$.  We endow $\Gr_{0}(\cP_{\bullet})$ with the following multiplication
$$
x \wedge y = \begin{tikzpicture}[baseline = .1cm]
	\draw (0, 0)--(0, .8);
	\draw (1.2, 0)--(1.2, .8);
	\nbox{unshaded}{(0,0)}{.4}{0}{0}{$x$}
	\nbox{unshaded}{(1.2,0)}{.4}{0}{0}{$y$}
	\node at (-.2, .6) {\scriptsize{$n$}};
	\node at (1, .6) {\scriptsize{$m$}};
\end{tikzpicture}
$$
for $x \in \cP_{n}$ and $y \in \cP_{m}$.  $\Gr_{0}(\cP_{\bullet})$ is endowed with the Voiculescu trace
$$
\tr(x) = \begin{tikzpicture}[baseline = .1cm]
	\draw (0, 0)--(0, .8);
	\nbox{unshaded}{(0,0)}{.4}{0}{0}{$x$}
	\nbox{unshaded}{(0,1.2)}{.4}{0}{0}{$\TL$}
	\node at (-.2, .6) {\scriptsize{$n$}};
\end{tikzpicture}
$$
where $\TL$ represents the sum of all Temperley-Lieb diagrams with all strings having endpoints at the bottom of the box.  We have the following facts about $\Gr_{0}(\cP)$.

\begin{thm}\label{thm:GJS}
\be

\item $\tr$ is positive definite on $\Gr_{0}(\cP_{\bullet})$, and the action of $x \in \Gr_{0}(\cP_{\bullet})$ on $\Gr_{0}(\cP_{\bullet})$ by left multiplication extends to an action on $L^{2}(\cP_{\bullet}, \tr)$ by bounded operators \cite{MR2732052}.

\item Let $\cM$ be the von-Neumann algebra generated by the action of $\Gr_{0}(\cP_{\bullet})$ on $L^{2}(\Gr_{0}(\cP_{\bullet}), \tr)$.  Then $\cM$ is a II$_{1}$ factor \cite{MR2732052}.  If $\cP_{\bullet}$ is finite depth, then $\cM \cong L(\F_{1 + 2(\delta - 1)I})$ where $I$ is the global index of $\cP_{\bullet}$ and $\delta$ is the loop paramater \cite{MR2807103}.  If $\cP_{\bullet}$ is infinite-depth, then $\cM \cong L(\F_{\infty})$ \cite{MR3110503}. 

\item Let $\cB$ be the C$^{*}$-algebra generated by the action of $\Gr_{0}(\cP_{\bullet})$ on $L^{2}(\Gr_{0}(\cP_{\bullet}), \tr)$.  If $\Gamma$ is the principal graph of $\cP$, then $K_{0}(\cB) \cong \Z\set{[\alpha]}{\alpha \in V(\Gamma)}$, moreover $\cB$ is isomorphic to $p_{\star}\cS(\Gamma, \mu)p_{\star}$ with $\star$ the unique depth-zero vertex of $\Gamma$ and $\mu$ the induced weighting from $\cP_{\bullet}$ \cite{MR3266249}.  If $\cA$ is the $*$-algebra generated by $\set{p_{\alpha}, \, X_{\e}}{\alpha \in V(\Gamma), \text{ and } \e \in E(\vec{\Gamma})}$, then this isomorphism carries $\Gr_{0}(\cP_{\bullet})$ onto $p_{\star}\cA p_{\star}$.

\ee
\end{thm}

Our work on the laws of elements in $\cA$ immediately implies the following corollary.

\begin{cor}
If $x = x^{*} \in \Gr_{0}(\cP_{\bullet})$ is not a scalar, then the Cauchy transform of $x$ with respect to $\tr$ is algebraic, the law of $x$ with respect to $\tr$ has no atoms, and the spectrum of $x$ is a finite union of closed intervals, each of which having measure in $(0, 1]  \cap \Z\set{\mu(\alpha)}{\alpha \in V(\Gamma)}$
\end{cor}

\begin{proof}
Since $\star$ is of minimal weight in $V(\Gamma)$, this is immediate from statement (3) of Theorem \ref{thm:GJS} as well as Corollary \ref{cor:alg}.

\end{proof}
\bibliographystyle{amsalpha}

{\footnotesize
\bibliography{bibliography}
}
\end{document}